\documentclass[onefignum,onetabnum]{siamart190516}

\usepackage{cite} % for interval of number citations
\usepackage{amsfonts}
\usepackage{graphicx,booktabs}
\usepackage{epstopdf}
\usepackage{algorithmic}
\ifpdf
  \DeclareGraphicsExtensions{.eps,.pdf,.png,.jpg}
\else
  \DeclareGraphicsExtensions{.eps}
\fi

\newsiamremark{remark}{Remark}
\newsiamremark{hypothesis}{Hypothesis}
\crefname{hypothesis}{Hypothesis}{Hypotheses}
\newsiamthm{claim}{Claim}

\headers{Equations with infinite delay: pseudospectral approximation}{F. Scarabel and R. Vermiglio}

\title{Equations with infinite delay:\\ pseudospectral discretization for numerical stability and bifurcation in an abstract framework%
\thanks{%
\textbf{Funding:}
The work of RS is supported by the Italian Ministry of University and Research (MUR) through the PRIN 2020 project (No. 2020JLWP23) ``Integrated Mathematical Approaches to Socio–Epidemiological Dynamics'', Unit of Udine (CUP: G25F22000430006).
}}

\author{Francesca Scarabel\thanks{Department of Applied Mathematics, University of Leeds, Woodhouse, Leeds LS2 9JT, UK (\email{f.scarabel@leeds.ac.uk}).}
\and Rossana Vermiglio\thanks{Department of Mathematics and Computer Science, University of Udine, via delle Scienze 208, 33100 Udine, Italy (\email{rossana.vermiglio@uniud.it}).}
}

\usepackage{amsopn}
\newcommand{\diff}{\mathop{}\!\mathrm{d}}

\usepackage{enumitem}

\begin{document}

\maketitle
\begin{abstract} %[not exceeding 250 words]
We consider nonlinear delay differential and renewal equations with infinite delay.
We extend the work of Gyllenberg \emph{et al}, \emph{Appl.\,Math.\,Comput.}\,(2018) by introducing a unifying abstract framework, and derive a finite-dimensional approximating system via pseudospectral discretization.
For renewal equations, we consider a reformulation in the space of absolutely continuous functions via integration.
We prove the one-to-one correspondence of equilibria between the original equation and its approximation, and that linearization and discretization commute. 
Our most important result is the proof of convergence of the characteristic roots of the pseudospectral approximation of the linear(ized) equations when the collocation nodes are chosen as the family of scaled zeros or extrema of Laguerre polynomials. 
This ensures that the finite-dimensional system correctly reproduces the stability properties of the original linear equation if the dimension of the approximation is large enough.
The result is illustrated with several numerical tests, which also demonstrate the effectiveness of the approach for the bifurcation analysis of equilibria of nonlinear equations.
The new approach used to prove convergence also provides the exact location of the spectrum of the differentiation matrices for the Laguerre zeros and extrema, adding new insights into properties that are important in the numerical solution of differential equations by pseudospectral methods.

\end{abstract}

% REQUIRED

% REQUIRED

\section{Introduction}\label{S_intro}
According to a standard definition, ``a delay equation is a rule for extending a function of time towards the future on the basis of the (assumed to be) known past'' \cite{DV21twin}.
Delay equations (DEs) include renewal equations (REs), which specify the value of the unknown function in terms of its past values, delay differential equations (DDEs), which specify the derivative, and systems coupling both \cite{DGG07,DiekmannBook,DV21twin}.

In this paper we focus on equations with infinite delay (iDEs, including iDDEs and iREs), which are widely used in mathematical biology, and in particular in ecology and epidemiology, to describe physiologically structured population models, in which the individual rates are assumed to depend on a structuring variable (e.g., age, size, time since infection) that evolves in time \cite{DGM18,DGM20,DGMND10,MetzDiekmannBook,SPOW21}.
In behavioral epidemic models, the infinite delay can arise to incorporate the effect of information and waning memory \cite{Donofrio2022, Manfredi2013}.
Other applications of iDDEs can be found in \cite{LiuMagal2020} and the references therein.

Our interest focuses on the long-term behavior of the systems, and in particular on the existence and stability of equilibria and their bifurcation analysis (i.e., how these properties change under parameter variation).
Differently from ordinary differential equations (ODEs), DEs generate infinite-dimensional dynamical systems where the state space is a suitable function space, hence substantially increasing the complexity of their theoretical and numerical analysis. 
One way to overcome theoretically the infinite dimensionality of the problem is via reduction to finite-dimensional inertial or center manifolds \cite{Foias1988,HaleBook}. 
However in this paper we take a different approach, aiming at studying \emph{numerically} the infinite-dimensional system by reducing it to a finite-dimensional system of ODEs, whose properties can be studied with well-established software, see e.g.~\cite{BDGSV16} and references therein. While in some special cases this reduction can be done exactly, e.g., via the linear chain trick 
\cite{CGHD21,DGM18,DGM20}, for more general kernels one usually relies on approximations. 

Among the various techniques available for this reduction \cite{BoydBook,GLZ06,ShenBook}, pseudospectral discretization (PSD) has been successfully applied in the last decade both to linear(ized) problems for the stability analysis of equilibria, and to nonlinear ones for the numerical bifurcation in various contexts, including DEs with bounded delay and PDEs with nonlocal boundary conditions \cite{ADLS22,BDR22,BDGSV16,BDMV13,BMVBook,DSV20,DSV2024,DeWolff21,SBDGV21,SDV21}. 
For DEs, the main advantage of using PSD compared to other approximation techniques is the fast (exponential) order of convergence for the stability of the equilibria, which follows from the fact that the eigenfunctions of the infinitesimal generator of the linear(ized) equation are exponential, hence analytic. 

To treat the unbounded delay, several approaches are used in the literature, including domain truncation or suitable change of variables \cite{BoydBook,GLZ06,ShenBook}, or even quadrature rules to approximate an iDDE with a DDE with a finite number of point delays to be studied via established methods (e.g.~\cite{BDGSV16}) or software like DDE-BIFTOOL \cite{ddebiftool0}. 
However, these approaches do not exploit some of the fundamental features of iDEs, in particular the ``exponentially fading memory'', which is the major advantage of the PSD proposed here. 
In fact, the step from finite to infinite delay involves several complexities, both from a theoretical and numerical point of view.

From a theoretical point of view, for iDEs the principle of linearized stability (PLS) has been proved in a space of exponentially weighted functions defined on $\mathbb{R}_- := (-\infty,0]$, with weight 
\begin{equation}\label{weight}
w(\theta):=\mathrm{e}^{\rho \theta}, \quad \theta \in \mathbb{R}_-,
\end{equation}
for some $\rho>0$.  
Using a modification of the sun-star calculus, Diekmann and Gyllenberg show that the stability of equilibria of a nonlinear equation is determined by the part of the spectrum of the infinitesimal generator of the linearized semigroup in the right-half plane 
\begin{equation*}
\mathbb{C}_\rho:=\{ \lambda \in \mathbb{C} \ | \  \Re \lambda  > -\rho\},
\end{equation*}
which contains eigenvalues only, that are roots of a characteristic equation \cite[Theorem 3.15 and Theorem 4.7]{DG12}. 
Recently, the PLS and the Hopf bifurcation theorem have been proved for iDDEs by applying the integrated semigroup theory in \cite{LiuMagal2020}. We also refer to \cite{S87} for a Hopf result in a space of bounded uniformly continuous functions. However, in this paper we keep reference to \cite{DG12} as it covers iDDEs, iREs and coupled iDDEs and iREs.
Finally, we note that other weight functions can be considered \cite{JSW87}, but here we restrict to the exponential weight \eqref{weight} for which the PLS has been rigorously proved. 

In this theoretical framework and motivated by the accuracy, effectiveness and flexibility shown by the PSD approach for bounded delay, a detailed experimental analysis has been carried out using collocation on the scaled Laguerre nodes in \cite{GSV18} for iDDEs.

Here, we introduce a unifying abstract differential equation (ADE) able to capture iDDEs and iREs within the same framework. 
This pragmatic representation allows us to encompass all the fundamental concepts of iDEs and the general derivation of the finite-dimensional dynamical system through the weighted version of the PSD all at once. 
For iREs, we also propose a new approach inspired by \cite{DV21twin,SDV21}: via integration, we map the original iRE, defined in a space of weighted integrable functions, to an equivalent equation defined in a subspace of absolutely continuous functions on any bounded subinterval of $\mathbb{R}_{-}$.
The main advantages are twofold: point evaluation, as well as polynomial interpolation and collocation, are well defined in the new space, and the domain of the infinitesimal generator is defined by a trivial condition, simplifying the implementation.

%We refer the interested reader e.g. to \cite{Kellerman1989,Thieme1990} for the theory of integrated semigroup, and again to \cite{LiuMagal2020} for the application of the theory to iDDEs, while other classes of functional differential equations are considered in  \cite{ }
%applications to iDDEs functional differential equations to \cite{}
%\bibitem{Kellerman1989}
%H. KELLERMANN AND M. HIEBER, Integrated semigroups, J. Funct. Anal. 1989
%\bibitem{Magal2018}
%P. Magal and S. Ruan, Theory and Applications of Abstract Semilinear Cauchy Problems, Springer- Verlag, 2018

Most importantly, here we add the key element to the numerical investigations in \cite{GSV18}: the rigorous proof that the PSD accurately approximates the stability properties of the equilibria. To this aim, we first show that there is a one-to-one correspondence between the equilibria of the ADE and its PSD. 
Then, we focus on the linear(ized) problems and prove the convergence of the stability indicators for the zero equilibrium, i.e., the eigenvalues of the infinitesimal generator in $\mathbb{C}_\rho$. 
Due to the complexities introduced by the step from finite to infinite delay and in the absence of suitable interpolation results on the semi-unbounded interval, the proof follows a slightly different route compared to the standard ones for finite delay \cite{BMVBook,SDV21}. 
The new approach guarantees the convergence of the approximated eigenvalues by using as collocation nodes the zeros or extrema of standard Laguerre polynomials suitably scaled.
Moreover, the general results for the Laguerre polynomials presented in the appendix provide an exact description of the location of the eigenvalues of the differentiation matrix with boundary condition built on the scaled Laguerre zeros and extrema.

The paper is organized as follows. 
In \cref{S_ADDE} we recall the basic results of stability theory for iDEs and introduce their abstract formulation. 
The unifying abstract framework is introduced in \cref{S_ADE}, together with the PSD approach. In \cref{S_convergence}, after deriving the characteristic equations, we show the one-to-one correspondence of equilibria of the original ADE and its finite-dimensional approximation, and prove the convergence of the characteristic roots. To improve readability, the results used in the proofs are collected in \cref{S_appendix} for the zeros and extrema of the standard Laguerre polynomials in $\mathbb{R}_+$. 
Finally, \cref{S_Results} illustrates by numerical experiments the theoretical convergence of the eigenvalues for linear iDEs and the flexibility of the approach for numerical bifurcation analysis of equilibria of nonlinear iDEs. 

\section{Abstract formulation of nonlinear iDEs}\label{S_ADDE}
In this section we recall from \cite{DG12} some basic results for iDDEs and iREs, including the PLS for the stability of equilibria. In view of \cref{S_ADE}, we also show how both iDDEs and iREs (and, by combining the two approaches, systems coupling both) can be formally written as semilinear ADEs.
Although we typically consider real-valued nonlinear equations and real-valued solutions, to apply spectral theory we assume that real Banach spaces are embedded into complex ones and all operators are extended to complex spaces. For further details on the complexification process we refer to \cite[Chapter III.7]{DiekmannBook}.

Let $\rho>0$ and consider the weight function \eqref{weight}. 
For $d \in \mathbb{N}$, $d>0$, we denote with $|\cdot|$ a norm in $\mathbb{C}^d$. 
Let $L^\infty_{\rho}:=L^\infty_{\rho}(\mathbb{R}_{-};\mathbb{C}^{d})$ be the space of (the equivalence classes of) functions $\varphi$ such that $|w\varphi|$ is essentially bounded in $\mathbb{R}_-$, and $L^p_{\rho}:=L^p_{\rho}(\mathbb{R}_{-};\mathbb{C}^{d})$, $1 \leq p < +\infty$, that of $\varphi$ such that $w|\varphi|^p$ is integrable, equipped with the norm 
$$
\|\varphi\|_{p,\rho}:= \|w\varphi\|_p =
\begin{cases}
\left( \int_{-\infty}^0 w(\theta) |\varphi(\theta)|^p \diff \theta \right)^{\frac{1}{p}}, & 1 \leq p < +\infty, \\[5pt]
{\rm ess} \sup_{\theta \in \mathbb{R}_{-}} w(\theta)|\varphi(\theta)|, & p=\infty. 
\end{cases}
$$
Note that $w L^p_{\rho} \cong L^p$, but we will often write $w$ explicitly to keep reference to the original state space of the dynamical system. 

To ease notation, we write $w\psi(\theta)$ in place of $w(\theta)\psi(\theta)$.  
We consider column vectors and use the convention of writing them as rows, but separating each column block-vector elements with semicolons.
We will also often abuse notation by using a symbol to denote both a constant function and the value it takes on.

%%%%%%%%%%%%%%%%%%%%%%%%%%%%%%%%%
%%%%%%%%%%%%%%%%%%%%%%%%%%%%%%%%%
\subsection{Delay differential equations}\label{ss_DDE}
A $d$-dimensional nonlinear iDDE is
\begin{equation}\label{DDE}
\frac{\diff y}{\diff t} (t) = F(y_t), \quad t\geq0,
\end{equation}
where $y_t(\theta):=y(t+\theta)$, $\theta \in \mathbb{R}_{-}$, is the ``history function'' at time $t\geq0.$ 
We assume that $F \colon \mathcal{C}_{\rho} \to \mathbb{R}^d$ is a continuously Fr\'echet differentiable map, where $\mathcal{C}_{\rho}:=\mathcal{C}_{0,\rho}(\mathbb{R}_{-};\mathbb{R}^{d})$ is the Banach space of all functions $\psi \colon \mathbb{R}_{-} \to \mathbb{R}^d$ such that $w\psi$ is continuous and vanishes at $-\infty$, equipped with the norm $\|\cdot\|_{\infty,\rho}$. 

To represent \eqref{DDE} in the abstract framework, we consider the space 
$$X:=w\mathbb{C}^d \times wL_\rho^\infty $$ 
equipped with the norm $\|(w\alpha;w\psi)\|_X=\max\{|\alpha|, \|w\psi\|_{\infty}\}.$ 
By defining the embedding $j\colon \mathcal{C}_{\rho} \to X$ as 
$j\psi:=(w\psi(0);w\psi)$, we identify $\mathcal{C}_{\rho}$ with the subspace $Y:=j(\mathcal{C}_{\rho})$ of $X$.
Note that, by interpreting $\alpha \in \mathbb{C}^d$ as a constant function in $L_\rho^\infty,$ we have that $X \subset w L_\rho^\infty(\mathbb{R}_-,\mathbb{C}^{\bar{d}})$ with $\bar{d}=2d$. 

By defining $u(t):=jy_t,$ it is easy to verify that $u$ satisfies the semilinear ADE
\begin{equation}\label{ADDE}
\frac{\diff u}{\diff t} (t)=\mathcal{A}_0^{\text{DDE}}u(t)+\mathcal{F}^{\text{DDE}}(u(t)), \quad t \geq 0,
\end{equation}
where the linear unbounded operator $\mathcal{A}_0^{\text{DDE}}\colon D(\mathcal{A}_0^{\text{DDE}}) \to X$ is
\begin{equation}\label{A0DDE}
\begin{array}{l}
\mathcal{A}_0^{\text{DDE}}(w\alpha;w\psi)=\frac{\diff}{\diff\theta}(w\alpha;w\psi)-\rho(w\alpha;w\psi)=(0;w \frac{\diff \psi }{\diff \theta}),\\[8pt]
D(\mathcal{A}_0^{\text{DDE}})=\{ (w\alpha;w\psi) \mid \psi \text { is Lipschitz continuous, } \psi(0)=\alpha \}, % \deleted{\subset Y},
\end{array}
\end{equation}
and the nonlinear operator $\mathcal{F}^{\text{DDE}}\colon Y \to X$ is 
\begin{equation}\label{FDDE}
\mathcal{F}^{\text{DDE}}(w\alpha;w\psi)=(wF(j^{-1}(w\alpha,w\psi));0)=(wF(\psi);0).
\end{equation}
Here, $\mathcal{A}_0^{\text{DDE}}$ captures the ``shift'' of the solution, $\mathcal{F}^{\text{DDE}}$ captures the ``extension'' according to the rule defined by \eqref{DDE}, while both account for ``weight multiplication''. Note that the $\theta$-derivative in \eqref{A0DDE} should be interpreted as almost everywhere.

Now we consider the nonlinear operator $\mathcal{F}^{\text{DDE}}.$  The canonical basis $e_1,\dots,e_d$ of $ \mathbb{R}^d$ determines $d$ linearly independent vectors $\xi_i=(e_i;0)$, $i=1,\dots,d,$ and $\Xi_d:= \text{span}\{ w\xi_i \mid i=1,...,d \} \subset X.$ By defining $\mathcal{F}_i^{\text{DDE}}(w\alpha;w\psi)= F_i(\psi)$, we get the finite-rank representation 
$\mathcal{F}^{\text{DDE}}=w\sum_{i=1}^d \mathcal{F}_i^{\text{DDE}}\xi_i$. Note that $\mathcal{F}^{\text{DDE}}$ inherits the regularity properties of $F$: if $F$ is continuously Fr\'echet differentiable, so is $\mathcal{F}^{\text{DDE}}$. 

We are interested in the stability properties of an equilibrium $\bar{\psi}$ of \eqref{DDE}, which is a constant solution $\bar{\psi}$ such that $F(\bar{\psi})=0$. The linearization of \eqref{DDE} at $\bar{\psi}$ is
\begin{equation}\label{LDDE}
\frac{\diff y}{\diff t}(t) = DF(\bar{\psi})y_t,
\quad t\geq 0,
\end{equation}
where $DF(\bar{\psi})$ is the Fr\'echet derivative of $F$ at $\bar{\psi}$.

From \eqref{A0DDE}--\eqref{FDDE} we get that $\bar{u}=(w\bar{\psi};w\bar{\psi})$ is an equilibrium of \eqref{ADDE} and vice versa. Moreover by introducing the linear operator $\mathcal{L}^{\text{DDE}}(w\alpha;w\psi)=(wDF(\bar{\psi})\psi;0)$, the abstract formulation of \eqref{LDDE} is 
\begin{equation}\label{LADDE}
\frac{\diff u}{\diff t}(t)=\mathcal{A}_0^{\text{DDE}} u(t)+\mathcal{L}^{\text{DDE}} u(t), \quad t \geq 0.
\end{equation}
By the PLS \cite[Theorem 4.7]{DG12}, the stability of $\bar{\psi}$ is determined by the stability of the zero solution of \eqref{LDDE} as follows. 
\begin{theorem} \label{Th_PLSDDE}
Let $\bar{\psi}$ be an equilibrium of \eqref{DDE} and $\bar{u}=(w\bar{\psi};w\bar{\psi})$ the corresponding equilibrium of \eqref{ADDE}. Then 
\begin{itemize}[leftmargin=*]
\item if all the eigenvalues of $\mathcal{A}^{\text{DDE}}:=\mathcal{A}_0^{\text{DDE}}+\mathcal{L}^{\text{DDE}}$ in $\mathbb{C}_\rho$ have negative real part, then $\bar{u}$ is exponentially stable;
\item if at least one eigenvalue of $\mathcal{A}^{\text{DDE}}$ has positive real part, then $\bar{u}$ is unstable;
\item the eigenvalues of $\mathcal{A}^{\text{DDE}}$ in $\mathbb{C}_\rho$ correspond to roots of the characteristic equation
\begin{equation*}%\label{CEDDE}
\det (\lambda I_{d}-DF(\bar{\psi})\mathrm{e}^{\lambda \cdot})=0
\end{equation*}
in $\mathbb{C}_\rho$, which are at most finitely many.
\end{itemize}
\end{theorem}
Given an eigenvalue $\lambda\in \mathbb{C}_\rho$ of $\mathcal{A}^{\text{DDE}}$, its eigenfunction is of the form $(w\alpha;w\psi_\lambda)$, where $\psi_\lambda(\theta) := \mathrm{e}^{\lambda\theta}\,\alpha$, for $\alpha \in \mathbb{C}^{d}\setminus \{0\}$.

%%%%%%%%%%%%%%%%%%%%%%%%%%%%%
%%%%%%%%%%%%%%%%%%%%%%%%%%%%%
\subsection{Renewal equations}\label{ss_RE}
Consider now the $d$-dimensional nonlinear iRE 
\begin{equation}\label{RE} 
y(t) = F(y_t), \quad t>0,
\end{equation}
for $y_t \in L^1_{\rho}$, where $F\colon L^1_{\rho} \to \mathbb{R}^{d}$ is a continuously Fr\'echet differentiable map.
In particular, this assumption means that we only consider equations \eqref{RE} such that their linearization is a linear Volterra integral equations, ruling out neutral REs with point delays. 
%as $F$ should map functions in the same equivalence class to the same value.

To derive the abstract equation, we define the space $X:=wL^1_{\rho}$ and the embedding $j:L^1_{\rho} \to X$ given by
$j\varphi (\theta) := w \int_0^\theta \varphi(s) \diff s$, for $\theta \in \mathbb{R}_{-}$.
Note that 
$$\|j\varphi\|_X \leq \int_{-\infty}^0 \mathrm{e}^{\rho \theta} \int_\theta^0 |\varphi(s)| \diff s \diff \theta =\int_{-\infty}^0  \int_{-\infty}^s \mathrm{e}^{\rho \theta} \diff \theta \ |\varphi(s)| \diff s \leq \frac{\|w\varphi\|_1}{\rho},
$$
and $Y:=jL^1_{\rho}=\{w\varphi \in X \mid w\varphi=j \frac{\diff \varphi }{\diff \theta} \text{ and } w \frac{\diff \varphi }{\diff \theta} \in X \}$ coincides with the subspace $AC_0:=AC_0(\mathbb{R}_{-};\mathbb{C}^{d})$ of all absolutely continuous functions $w\varphi \in X$ in any closed subset of $\mathbb{R}_-$, such that $w\varphi(0)=0.$ Note that, for $w\varphi \in Y$, $j^{-1}(w\varphi)= \frac{\diff \varphi }{\diff \theta}$. \smallskip

By defining the state $u(t):=jy_t$, we can formally derive the ADE
\begin{equation}\label{ARE}
\frac{\diff u}{\diff t}(t)=\mathcal{A}_0^{\text{RE}}u(t)+\mathcal{F}^{\text{RE}}(u(t)), \quad t \geq 0,
\end{equation}
where
\begin{align}
&\mathcal{A}_0^{\text{RE}}(w\varphi)= \textstyle{\frac{\diff}{\diff \theta}}(w\varphi) -\rho w\varphi=w \textstyle{\frac{\diff \varphi }{\diff \theta}}, & & w\varphi \in D(\mathcal{A}_0^{\text{RE}})=Y,\label{A0RE} \\[8pt]
&\mathcal{F}^{\text{RE}}(w\varphi)=-wF(j^{-1}w\varphi)=-wF( \textstyle{\frac{\diff \varphi }{\diff \theta}}), & & w\varphi \in Y. \label{FRE}
\end{align}
Similarly as before, $\mathcal{F}^{\text{RE}}$ has finite rank with range $\Xi_d=w\mathbb{R}^d.$ 

As expected, there is a one-to-one correspondence between the equilibria $\bar{\varphi}$ of \eqref{RE}, which are constant functions such that $\bar{\varphi}=F(\bar{\varphi}),$ and the equilibria $\bar{u}=j\bar{\varphi}$ of \eqref{ARE}. Moreover, by defining $\mathcal{L}^{\text{RE}}:=-wDF(\bar{\varphi})j^{-1},$ the abstract formulation of the linearized RE 
$y(t) = DF(\bar{\varphi})y_t$, $t>0$, gives the linearized equation
\begin{equation}\label{LARE}
\frac{\diff u}{\diff t}(t) =\mathcal{A}_0^{\text{RE}} u(t)+\mathcal{L}^{\text{RE}}u(t), \quad t \geq 0.
\end{equation}
The following PLS holds \cite[Theorem 3.9]{DG12}.
\begin{theorem} \label{Th_PLSRE}
Let $\bar{\varphi}$ be an equilibrium of \eqref{RE} and $\bar{u}=j\bar{\varphi}$ the corresponding equilibrium of \eqref{FRE}. Then 
\begin{itemize}[leftmargin=*]
\item if all the eigenvalues of $\mathcal{A}^{\text{RE}}:=\mathcal{A}_0^{\text{RE}}+\mathcal{L}^{\text{RE}}$ in $\mathbb{C}_\rho$ have negative real part, then $\bar{u}$ is exponentially stable;
\item if at least one eigenvalue of $\mathcal{A}^{\text{RE}}$ has positive real part, then $\bar{u}$ is unstable;
\item the eigenvalues of $\mathcal{A}^{\text{RE}}$ in $\mathbb{C}_\rho$ correspond to zeros of characteristic equation
\begin{equation*}%\label{CERE}
\det ( I_{d}-DF(\bar{\varphi})\mathrm{e}^{\lambda \cdot})=0
\end{equation*}
in $\mathbb{C}_\rho,$ which are at most finitely many.
\end{itemize}
\end{theorem}
Given an eigenvalue $\lambda \in \mathbb{C}_\rho$ of $\mathcal{A}^{\text{RE}}$, the corresponding eigenfunction is $w\varphi_\lambda$, where $\varphi_\lambda(\theta) := \frac{\mathrm{e}^{\lambda\cdot}-1}{\lambda} \alpha$, for $\alpha \in \mathbb{C}^{d}\setminus \{0\}$. Note also that $j^{-1}(w\varphi_\lambda) = \mathrm{e}^{\lambda \cdot} \alpha$.

%%%%%%%%%%%%%%%%%%%%%%%%%%%%%
%%%%%%%%%%%%%%%%%%%%%%%%%%%%%
\section{A unifying abstract framework}\label{S_ADE}
We now propose a unifying abstract framework which is particularly useful to derive the approximating finite-dimensional system via PSD and to analyse the convergence in \cref{S_convergence}.
% Motivated by the previous discussions and by the fact that iDDEs are sometimes treated in the space $L^2_\rho$ \cite{ShenBook}, 
We consider a general framework where $X$ is a Banach space of $\mathbb{C}^{\bar{d}}$-valued functions defined on $\mathbb{R}_{-}$ equipped with a non-weighted $p$-norm, $1 \leq p \leq +\infty,$ denoted by $\|\cdot\|_X.$ In other words, $X$ is a closed subspace of $L^p(\mathbb{R}_-;\mathbb{C}^{\bar{d}}).$  
% Note that, for iDDEs, the authors of \cite{LiuMagal2020} work in the space of exponentially bounded and uniformly continuous functions. 

As before, let $\rho>0$ and $w$ be the function \eqref{weight}. Consider the semilinear ADE
\begin{equation}\label{ADE}
\frac{\diff u}{\diff t}(t)=\mathcal{A}_0 u(t)+\mathcal{F}(u(t)), \quad t \geq 0,
\end{equation}
where $\mathcal{A}_0\colon D(\mathcal{A}_0) \to X$ is the linear unbounded operator 
\begin{equation}\label{A0}
\begin{array}{l}
\mathcal{A}_0 x= \frac{\diff x}{\diff \theta}-\rho x=  w \frac{\diff }{\diff \theta} (\frac{x}{w}), \quad x \in D(\mathcal{A}_0), \\[8pt]
D(\mathcal{A}_0)=\{ x \in  X \mid \mathcal{A}_0x \in X \text{ and } x(0)=\mathcal{B}x \},
\end{array}
\end{equation}
with $\mathcal{B}\colon X \to \mathbb{R}^{\bar{d}}$ linear and bounded. 

\begin{remark}\label{remB}
The operators $\mathcal{A}_0^{\text{DDE}}$ and $\mathcal{A}_0^{\text{RE}}$ defined in \eqref{A0DDE} and \eqref{A0RE} are particular instances of \eqref{A0} with $\mathcal{B}(w\alpha,w\psi)=(\alpha;\alpha)$ and $\mathcal{B}(w\varphi)=0$, respectively.  
In particular, the form of $\mathcal{B}$ characterizes the type of iDE under investigation. 
\end{remark}

Note that $D(\mathcal{A}_0)$ is a subset of the Sobolev space $W^{p,1}(\mathbb{R}_-;\mathbb{C}^{\bar{d}}),$ hence it admits one (and only one) continuous representative \mbox{\cite{BrezisBook}}. 
We assume that $\mathcal{F}\colon D(\mathcal{A}_0) \to X$ is continuously Fr\'echet differentiable and admits the finite-rank representation 
\begin{equation}\label{F-finite-rank}
\mathcal{F}(x)=w\sum_{i=1}^d \mathcal{F}_i(x)\xi_i \in \Xi_d
\end{equation}
for some $\mathcal{F}_i \colon D(\mathcal{A}_0) \to \mathbb{R}$ sufficiently smooth, $\xi_i $, $i=1,\dots,d$, linearly independent vectors in $\mathbb{C}^{\bar{d}}$, and $\Xi_d:= w\, \text{span}\{ \xi_i \mid i=1,...,d \}.$ 
The continuity of the elements of $D(\mathcal{A}_0)$, where we assume that $u(t)$ lives for $t \geq 0$, justifies the choice of the PSD approach, which requires point evaluation to be well defined.

%%%%%%%%%%%%%%%%%%%%%%%%%%%%%
%%%%%%%%%%%%%%%%%%%%%%%%%%%%%
\subsection{Reduction to finite dimension via PSD}\label{ss_APSD}

Given a positive integer $N$, called \emph{discretization index}, we consider a set of $N$ distinct points in $\mathbb{R}_-$,
\begin{equation*} % \label{nodes}
\Theta_N:=\{ \theta_{1},\dots,\theta_{N} \mid \theta_{N}<\dots <\theta_{1} < 0\}.
\end{equation*}
We refer to the set $\Theta_N$ as the \emph{collocation nodes}.
Then, we represent every $x \in D(\mathcal{A}_0)$ by a set of vectors $X_j = x(\theta_j)$, $j=1,\dots,N$ via the {\it restriction operator} 
$$\mathcal{R}_Nx=(X_{1};\dots;X_{N}) \in \mathbb{C}^{N\bar{d}}.$$

The operator $\mathcal{R}_N$ is complemented with a {\it prolongation operator} $\mathcal{P}_{N,0}$ that maps $\mathbb{R}^{N\bar{d}}$ to the space $D(\mathcal{A}_0) \cap \Pi_{N}^{\bar{d}}$, where $\Pi_{N}^{\bar{d}}$ is the space of $\mathbb{C}^{\bar{d}}$-valued polynomials of degree at most $N$.
$\mathcal{P}_{N,0}$ is defined such that $x_{N,0}=\mathcal{P}_{N,0}(X_1;\dots;X_N) $ satisfies the implicit condition $x_{N,0}(0)=\mathcal{B}x_{N,0}$. 
In practice, we first define the extended mesh 
\begin{equation*}
\Theta_{N,0}:=\Theta_{N} \cup \{\theta_0=0\}
\end{equation*} 
and consider the $N+1$ Lagrange polynomials $\ell_{i,0}$, $i=0,\dots,N$, associated with $\Theta_{N,0}$.
Then, we take
$$ x_{N,0} = \beta \,\ell_{0,0} + \sum_{i=1}^N X_i \ell_{i,0}, $$
with $\beta$ implicitly defined by $x_{N,0}(0) = \mathcal{B} x_{N,0}$. 
With these choices, 
$\mathcal{A}_0\mathcal{P}_{N,0}$ maps into $w\Pi_{N-1}^{\bar{d}}$ and the application of $\mathcal{R}_N$ is well defined. 

Using $\mathcal{R}_N$ and $\mathcal{P}_{N,0}$ we can derive the discrete versions of $\mathcal{A}_{0}$ and $\mathcal{F}$, namely  
\begin{equation*}%\label{AFLN}
\mathcal{A}_{0,N}:=\mathcal{R}_N\mathcal{A}_{0}\mathcal{P}_{N,0}, \quad
\mathcal{F}_{N}:=\mathcal{R}_N\mathcal{F} \circ \mathcal{P}_{N,0},
\end{equation*}
where $\circ$ denotes composition of nonlinear operators. 
Finally, the PSD of \eqref{ADE} is the semilinear ODE
\begin{equation}\label{ODE}
\frac{\diff U}{\diff t} (t)=\mathcal{A}_{0,N} U(t)+\mathcal{F}_N(U(t)), \quad t \geq 0,
\end{equation}
for $U(t) \in \mathbb{C}^{N\bar{d}}$.
In the next section, we study how the equilibria and stability indicators of \eqref{ADE} are approximated by the counterparts of \eqref{ODE}.

%%%%%%%%%%%%%%%%%%%%%%%%%%%%%
%%%%%%%%%%%%%%%%%%%%%%%%%%%%%
\section{Approximation of equilibria and their stability}\label{S_convergence}
An equilibrium $\bar{u} \in D(\mathcal{A}_0)$ of \eqref{ADE} is a solution of
$ \mathcal{A}_0\bar{u}+\mathcal{F}(\bar{u})=0$, 
and, for $\mathcal{L}:=D\mathcal{F}(\bar{u})$, the linearization of \eqref{ADE} at $\bar{u}$ is 
\begin{equation}\label{LADE}
\frac{\diff u}{\diff t}(t)=\mathcal{A}_0 u(t)+\mathcal{L}u(t), \quad t \geq 0.
\end{equation} 
The linearized operator $\mathcal{L}$ inherits the finite-rank property and can be written as 
\begin{equation*}%\label{rankL}
\mathcal{L}x=w \sum_{i=1}^d (\mathcal{L}_ix) \xi_i, \quad  \text{for all } x\in D(\mathcal{A}_0),
\end{equation*}
for $\mathcal{L}_i \colon D(\mathcal{A}_0) \to \mathbb{C}$, $i=1,\dots,d$. We define 
$\mathcal{L}^\xi \colon D(\mathcal{A}_0) \to \mathbb{C}^d$ by $\mathcal{L}^\xi :=(\mathcal{L}_1;\dots;\mathcal{L}_d).$%

Inspired by the PLS for iDEs we assume the following throughout. 
\begin{hypothesis}\label{PLS}
The local stability of an equilibrium $\bar{u}$ of the semilinear equation \eqref{ADE} can be inferred from the stability of the zero solution of the linearized equation \eqref{LADE}, and 
the position of the eigenvalues of $\mathcal{A}_0+\mathcal{L}$ in $\mathbb{C}_{\rho}$ determines the stability of the zero solution of \eqref{LADE}.
\end{hypothesis}

\begin{proposition}
Let $\rho = \rho_1+\epsilon$ with $\rho_1,\epsilon>0$. Then, $\lambda \in \mathbb{C}_{\rho_1}$ is an eigenvalue of $\mathcal{A}_0 + \mathcal{L}$ if and only if $\lambda$ is a root of the characteristic equation
\begin{equation}\label{CE}
\det(I_{\bar{d}+d}-A(\lambda))=0,
\end{equation}
where $A(\lambda): \mathbb{C}^{\bar{d}+d} \to \mathbb{C}^{\bar{d}+d}$ is the linear operator
\begin{equation*}
A(\lambda)(\beta;\gamma)= \left(\mathcal{B} x(\cdot; \lambda,\beta,\gamma); \ 
\mathcal{L}^\xi x(\cdot; \lambda,\beta,\gamma) \right),
\end{equation*}
and $x(\cdot; \lambda,\beta,\gamma)$ is the solution of the initial value problem
\begin{equation}\label{IVP_gamma}
\begin{cases}
\displaystyle{\frac{\diff x}{\diff \theta}}=(\lambda +\rho)x+w \sum_{i=1}^d \gamma_i \xi_i, \\[5pt]
x(0)=\beta.
\end{cases}
\end{equation}
with $\beta \in \mathbb{C}^{\bar{d}}$ and $\gamma \in \mathbb{C}^d$. In this case, $x(\cdot; \lambda,\beta,\gamma) \in D(\mathcal{A}_0)$ is the eigenfunction corresponding to $\lambda$. 
\end{proposition}

\begin{proof}
Let $x = x(\cdot;\lambda,\beta,\gamma)$ be the solution of \eqref{IVP_gamma}. 
By variation of constants,
\begin{equation}\label{eigsol}
x(\cdot;\lambda,\beta,\gamma)=e(\lambda+\rho)\beta + h(\lambda+\rho)\gamma,
\end{equation}
where $e(\lambda+\rho)\colon \mathbb{C}^{\bar{d}} \to X$ and $h(\lambda+\rho)\colon \mathbb{C}^{d} \to X$ are defined by
\begin{align*}
e(\lambda+\rho)\beta (\theta) &:= \mathrm{e}^{(\lambda+\rho)\theta} \beta=w(\theta)\mathrm{e}^{\lambda \theta}\beta, \\
h(\lambda+\rho)\gamma (\theta) &:= \textstyle{\left(\sum_{i=1}^d \gamma_i \xi_i \right) \int_0^\theta \mathrm{e}^{(\lambda +\rho)(\theta-s)} w(s)\, \diff s  } , \quad \theta \in \mathbb{R}_-.
% =w(\theta) \frac{\mathrm{e}^{\lambda \theta}-1}{\lambda}\sum_{i=1}^d \gamma_i \xi_i
\end{align*}
Now, \eqref{eigsol} is an eigenfunction of $\mathcal{A}_0+\mathcal{L}$ corresponding to $\lambda$ if and only if there exists $(\beta;\gamma) \neq 0$ such that $\beta=\mathcal{B}x$ and $\gamma=\mathcal{L}^\xi x,$ i.e., if and only if $\lambda$ is a root of \eqref{CE}.
\end{proof}
The computation of (part of) the eigenvalues of $\mathcal{A}_0+\mathcal{L}$ in $\mathbb{C}_\rho$ is therefore turned into the computation of the roots of \eqref{CE}.
From \cref{PLS}, the position of the characteristic roots of \eqref{CE} in $\mathbb{C}_{\rho}$ determines the stability of the zero solution of \eqref{LADE}.

%\begin{remark} \label{r_eigsDDERE}
%\added{
%\eqref{CE} captures both \eqref{CEDDE} and \eqref{CERE}. 
%First recall \cref{ss_DDE} and \cref{remB} and note that $h(\lambda+\rho)\gamma = (\xi^T \gamma) \frac{\mathrm{e}^{\lambda\cdot}-1}{\lambda}$. 
%For iDDEs, let $x(\cdot;\lambda,\beta,\gamma) = (w\alpha;w\psi)$.
%Then it must be $\beta = (\alpha;\alpha)$ and, from \eqref{eigsol}, we conclude that $\gamma = -\alpha \lambda$ and $\psi = \alpha \mathrm{e}^{\lambda\cdot}$. 
%Finally, the condition $\xi^T \gamma=\mathcal{L} x$ is satisfied (i.e., $\lambda$ is an eigenvalue) if and only if $\lambda$ is a root of \eqref{CEDDE}, with corresponding eigenfunction $x = (w\alpha;w\alpha \mathrm{e}^{\lambda\cdot})$.  
%For iREs, let $x(\cdot;\lambda,\beta,\gamma) = w\varphi$; it is then easy to see that $\varphi(\theta) = \gamma \frac{\mathrm{e}^{\lambda\cdot}-1}{\lambda}$.
%Finally, $\xi^T \gamma=\mathcal{L} x$ holds if and only if $\lambda$ is a root of \eqref{CERE}.  
%}
%\end{remark}

\smallskip
Regarding the approximating system \eqref{ODE}, the PLS holds from the classical theory of ODEs. To approximate the stability of an equilibrium $\bar{u}$ of \eqref{ADE}, we first need the following results to relate the linearized equations.
\begin{theorem}[One-to-one correspondence of equilibria]\label{th:one-to-one}
Let $N\geq 1$. 
If  $\overline{u}\in D(\mathcal{A}_0)$ is an equilibrium of \eqref{ADE}, then $\overline{U} = \mathcal{R}_N \overline{u}$ is an equilibrium of \eqref{ODE}. Vice versa, if $\overline{U}$ is an equilibrium of \eqref{ODE}, then $\overline{u} = \mathcal{P}_{N,0} \overline{U}$ is an equilibrium of \eqref{ADE}.

Moreover, linearization at equilibrium and PSD commute. In other words, the linearization of \eqref{ODE} at $\overline{U}$ coincides with the PSD of \eqref{LADE} and reads
\begin{equation}\label{LODE}
\frac{\diff U}{\diff t} (t)=\mathcal{A}_{0,N} U(t) +\mathcal{L}_NU(t), \quad t \geq 0,
\end{equation}
with $\mathcal{L}_N = \mathcal{R}_N \mathcal{L} \mathcal{P}_{N,0}$.
\end{theorem}

\begin{proof}
First, if $\overline{u} \in D(\mathcal{A}_0)$ is an equilibrium of \eqref{ODE}, then $\overline{u}(0) = \mathcal{B}\overline{u}$. From \eqref{F-finite-rank} and the equality $\mathcal{A}_0\overline{u}+\mathcal{F}(\overline{u})=0$, integrating we have 
$$\frac{\diff }{\diff \theta}\left(\frac{\overline{u}}{w}\right) = - \frac{\mathcal{F}(\overline{u})}{w} = - \sum_{i=1}^d  \mathcal{F}_i(\overline{u}) \xi_i \in \mathbb{C}^{\bar{d}} \quad \Rightarrow \quad \overline{u}(\theta) = w(\theta) (\mathcal{B}\overline{u} - \theta \sum_{i=1}^d  \mathcal{F}_i(\overline{u}) \xi_i). $$
%$$(w^{-1} \overline{u})' = - w^{-1} \mathcal{F}(\overline{u}) = - \sum_{i=1}^d  \mathcal{F}_i(\overline{u}) \xi_i \in \mathbb{C}^{\bar{d}} \quad \Rightarrow \quad \overline{u}(\theta) = w(\theta) (\mathcal{B}\overline{u} - \theta \sum_{i=1}^d  \mathcal{F}_i(\overline{u}) \xi_i). $$
Since $\overline{u}/w$ is a polynomial of degree one and $ \overline{u}(0) = \mathcal{B}\overline{u}$, defining $\overline{U} := \mathcal{R}_N \overline{u}$, we have $\mathcal{P}_{N,0} \overline{U} = \overline{u}$. Therefore
$$ \mathcal{A}_{0,N}\overline{U}+\mathcal{F}_N(\overline{U})=\mathcal{R}_N (\mathcal{A}_0 \overline{u} + \mathcal{F}(\overline{u})) = 0,$$
hence $\overline{U}$ is an equilibrium of \eqref{ODE}. 
Vice versa, let $\overline{U}$ be an equilibrium of \eqref{ODE} and let $\overline{u}:= \mathcal{P}_{N,0}\overline{U} \in D(\mathcal{A}_0).$ Note that $(\mathcal{A}_0 \overline{u})/w$ is a polynomial of degree at most $N-1$, and the function 
$((\mathcal{A}_0\overline{u})/w +\sum_{i=1}^d  \mathcal{F}_i(\overline{u}) \xi_i)$ is zero on $\Theta_N$. Therefore, it must be zero as a polynomial, and the polynomial $\overline{u}$ is in fact an equilibrium of \eqref{ADE}. 

Let us now consider the linearization at $\overline{u}$. Trivially, \eqref{LODE} is the PSD of \eqref{LADE}. 
On the other hand, for the linearity of $\mathcal{R}_N$ and $\mathcal{P}_{N,0}$ we have \\[5pt]
{\centering 
$ D (\mathcal{R}_N \mathcal{F} \mathcal{P}_{N,0} ) (\overline{U}) = \mathcal{R}_N D (\mathcal{F} \mathcal{P}_{N,0} ) (\overline{U}) = (\mathcal{R}_N D\mathcal{F} (\mathcal{P}_{N,0}\overline{U}) \mathcal{P}_{N,0}) \overline{U} = \mathcal{R}_N \mathcal{L} \mathcal{P}_{N,0}  \overline{U}. 
$} \medskip
\hfill\qedsymbol{}
\end{proof}
For \eqref{LODE} we derive a discrete characteristic equation following \cite[Proposition 5.2]{BMVBook}.
\begin{proposition}\label{PCEN}
The eigenvalues of $\mathcal{A}_{0,N} +\mathcal{L}_N$ in $\mathbb{C}_\rho$ are the roots of the discrete characteristic equation
\begin{equation}\label{CEN}
\det(I_{\bar{d}+d}-A_N(\lambda))=0,
\end{equation}
where $A_N(\lambda): \mathbb{C}^{\bar{d}+d} \to \mathbb{C}^{\bar{d}+d}$ is the linear operator
\begin{equation*}
A_N(\lambda)(\beta;\gamma)=(\mathcal{B}x_{N,0}(\cdot;\lambda,\beta,\gamma);\mathcal{L}x_{N,0}(\cdot;\lambda,\beta,\gamma)),
\end{equation*}
and $x_{N,0}(\cdot;\lambda,\beta,\gamma) \in w\Pi_N^{\bar{d}}$ is the solution of the collocation equation
\begin{equation}\label{coll_gamma}
\begin{cases}
\displaystyle{\frac{\diff x_{N,0}}{\diff \theta}}(\theta_{i})=(\lambda +\rho)x_{N,0}(\theta_{i})+ w(\theta_{i})\sum_{i=1}^d \gamma_i\xi_i\, \quad i=1,\dots,N,\\[5pt]
x_{N,0}(0)=\beta.
\end{cases}
\end{equation} 
\end{proposition}

%%%%%%%%%%%%%%%%%%%%%%%%%%%%%
%%%%%%%%%%%%%%%%%%%%%%%%%%%%%
\subsection{Convergence of the approximation of characteristic roots}\label{SS_LADE}
We now show that the roots of \eqref{CEN} converge to the roots of \eqref{CE} as $N\to\infty$. We first assume that the operator $\mathcal{L}$ can be computed exactly and so its PSD is given by $\mathcal{L}_N = \mathcal{R}_N \mathcal{L} \mathcal{P}_{N,0}$. We later briefly discuss the impact of quadrature rules.

Here we take $\rho=\rho_1+\epsilon$ for some $\rho_1,\epsilon>0$, and we choose
\begin{equation} \label{scaled_zeros}
\theta_j := -\frac{t_{j,N}}{2\rho_1}, \quad j=1,\dots,N,
\end{equation}
where $t_{j,N}$, $j=1,\dots,N$ are either the zeros or the extrema of the standard Laguerre polynomials in $\mathbb{R}_+$ (see \cref{S_appendix} for details). We refer to \eqref{scaled_zeros} as the \emph{scaled Laguerre zeros} or \emph{extrema}. 

\begin{theorem}\label{Th00}
Let $\rho=\rho_1+\epsilon$ with $\rho_1,\, \epsilon>0$, and $B$ be an open ball in $\mathbb{C}_{\rho_1}$. 
For $\Theta_N$ chosen as the scaled Laguerre zeros or extrema \eqref{scaled_zeros}, the collocation equation \eqref{coll_gamma} admits a unique solution $x_{N,0}(\cdot;\lambda,\beta,\gamma) \in w\Pi_N^{\bar{d}}$ for all $N \geq 1$, $\lambda \in B$, $\beta \in \mathbb{C}^{\bar{d}}$ and $\gamma \in \mathbb{C}^d.$ 
Moreover, let $x(\cdot;\lambda,\beta,\gamma)$ be the solution of \eqref{IVP_gamma} and let
\begin{equation*}%\label{CCl}
C(\lambda):=\frac{\lambda}{\lambda+2\rho_1}, \quad \text{and } \quad K^p_\epsilon = 
\begin{cases}
	1 & \text{ for } p=\infty,\ \epsilon \geq 0,\\
  	(\frac{p\epsilon}{2\rho_1})^{-1/p} & \text{ for } 1 \leq p <\infty,\ \epsilon > 0.
\end{cases}
\end{equation*} 
The error 
$e_{N} = x_{N,0}(\cdot; \lambda,\beta,\gamma) - x(\cdot; \lambda,\beta,\gamma)$ satisfies 
\begin{equation*}
 \|e_{N}\|_X \leq {C}_0(B,\epsilon) D_N(B)\, |(\beta;\gamma)|\\
\end{equation*}
where for the scaled Laguerre zeros
$$ C_0(B,\epsilon) := k K^p_\epsilon \max_{\lambda \in \overline B} \left\{ \frac{\max\{|\lambda|,1\}}{\Re \lambda +\rho_1} \right\} \quad \text{ and }\quad 
D_N(B) :=\max_{\lambda \in \overline B} |C(\lambda)|^N,
$$
while for the scaled Laguerre extrema
$$
C_0(B,\epsilon) :=k K^p_\epsilon \, \max_{\lambda \in \overline B} \left\{2 + \frac{|\lambda|}{2\rho_1(\Re \lambda+\rho_1)}\right\} \quad \text{ and } \quad 
D_N(B) :=\max_{\lambda \in \overline B} \Big| \frac{C(\lambda)^{N}}{1- C(\lambda)^{N+1}} \Big|, 
$$
where $k$ is a positive constant.
Finally, since $|C(\lambda)|<1$, $\|e_N\|_X \to 0$ as $N \to \infty$.
\end{theorem}
\begin{proof}
First we observe that, for the solution of \eqref{IVP_gamma}, we can write $x=w\varphi \in L^p$ with $\varphi \in L^p_\rho$ solution of 
\begin{equation*}
\begin{cases}
\displaystyle{\frac{\diff \varphi}{\diff \theta}} = \lambda \varphi + \textstyle{\sum_{i=1}^d \gamma_i\xi_i}  & \theta \leq 0,\\[5pt]
\varphi(0) = \beta.
\end{cases}
\end{equation*}
Second, by defining 
\begin{equation*}
	\theta \mapsto t := - 2\rho_1 \theta, \quad 
	\lambda \mapsto \mu:= - \frac{\lambda}{2\rho_1}, \quad \text{and} \quad  
	\sum_{i=1}^d \gamma_i\xi_i\mapsto c := - \frac{1}{2\rho_1}\sum_{i=1}^d \gamma_i\xi_i,
\end{equation*}
we transform the problem for $\varphi$ on $\mathbb{R}_-$ with $\Re \lambda > -\rho_1$ into the problem 
\begin{equation*}
\begin{cases}
\displaystyle{\frac{\diff y}{\diff \theta}} = \mu y+ c,  & t \geq 0,\\[5pt]
y(0) = \beta,
\end{cases}
\end{equation*} 
for $y$ on $\mathbb{R}_+$ with $\Re \mu < \frac{1}{2}.$ Under this mapping, $\mathrm{e}^{\rho_1\theta}=\mathrm{e}^{-t/2},$ while $w(\theta) = \mathrm{e}^{\rho\theta} = \mathrm{e}^{-(1/2+\delta)t}$ with $\delta := \frac{\epsilon}{2\rho_1}$, and the collocation nodes $\theta_{j}$ are associated to the zeros or extrema $t_j$ of the standard Laguerre polynomials in $\mathbb{R}_+.$ 
Now the result follows by applying \cref{pr:existence_zeros} and \cref{pr:existence_extrema} in \cref{S_appendix} on the positive real line $\mathbb{R}_+$ and mapping  the bounds \eqref{eNLX} and \eqref{eNRLX} back to $\mathbb{R}_-.$
\end{proof}
 
Finally, from the error bounds on the collocation equation we are ready to conclude the error bounds on the characteristic roots in the following theorem.
\begin{theorem}\label{Th2}
Let $\rho=\rho_1+\epsilon$, with $\rho_1, \epsilon>0,$ and $\lambda^* \in \mathbb{C}_{\rho_1}$ be a root of \eqref{CE} with multiplicity $m$. Let $B$ be an open ball in $\mathbb{C}_{\rho_1}$ of center $\lambda^*$ such that $\lambda^*$ is the unique root of \eqref{CE} in $B$. 
For $\Theta_N$ chosen as the scaled Laguerre zeros or extrema \eqref{scaled_zeros}, there exists $N_1:=N_1(B)$ such that, for $N\geq N_1$, equation \eqref{CEN} has $m$ roots $\lambda_{i,N}$, $i=1,\dots,m$ (taking into account multiplicities) and 
\begin{equation}\label{lambda_error}
\max_{i=1,\dots,m}| \lambda^*-\lambda_{i,N}|=O\left(D_N(B)^{1/m} \right).
\end{equation}
Hence $\displaystyle{\max_{i=1,\dots,m}} | \lambda^*-\lambda_{i,N}| \to 0$ as $N\to \infty$.
\end{theorem}
\begin{proof}
The steps follow the lines of \cite[Section 5.3.2]{BMVBook}, to which we refer to further details.
From \cref{Th00}, for all choices of the $p$-norm and for both choices of the nodes, for all $\lambda \in B$ we can bound
\begin{align*}
\|(A(\lambda) - A_N(\lambda))(\beta;\gamma)\|_X 
	% &=  \|(\mathcal{B},\mathcal{L}^{\xi}) (x(\cdot;\lambda,\beta,\gamma) - x_{N,0}(\cdot;\lambda,\beta,\gamma)) \|_X \\
	&\leq \|(\mathcal{B},\mathcal{L}^{\xi})\| \cdot \|e_N \|_X %\\
	\leq \|(\mathcal{B},\mathcal{L}^{\xi})\| \cdot C_0(B,\epsilon) D_N(B) |(\beta;\gamma)|,
\end{align*}
with $C_0(B,\epsilon)$, $D_N(B)$ obtained from \cref{Th00}.
Therefore, considering the operator norm in $\mathbb{C}^{\bar{d}+d}$, following the proof of \cite[Proposition 5.3]{BMVBook} we obtain the norm bound
$$ \|(A(\lambda) - A_N(\lambda))\| \leq C_1(B) C_0(B,\epsilon)  D_N(B)
$$
for $C_1(B) := \max \left\{ \left\| (\det)' (I_{\bar{d}+d} - A(z) + \Delta ) \right\|  \colon  z \in \overline{B}, \ \Delta \in \mathbb{C}^{\bar{d}+d}, \ \|\Delta\|\leq 1 \right\}. $
%$$C_3(B) = \max_{\substack{z \in \overline{B} \\ \Delta \in \mathbb{C}^{\bar{d}+d} \\ \|\Delta\|\leq 1}} \left\| (\det)' (I_{\bar{d}+d} - A(z) + \Delta ) \right\|. $$
This allows us to conclude the proof using Rouch\'e's Theorem from complex analysis \cite[Section 7.7]{Prietsley1990}. 
\end{proof}

The approximations \eqref{ODE} and \eqref{LODE} hold when $\mathcal{F}$ and $\mathcal{L}$ can be computed exactly. However, such operators often involve integrals, hence need to be approximated by $\widetilde{\mathcal{F}}$ and $\widetilde{\mathcal{L}}$ via suitable quadrature rules. In this case, \eqref{ODE}, \eqref{LODE} and \eqref{CEN} still hold with the operators defined by $\mathcal{F}_N= \mathcal{R}_N\widetilde{\mathcal{F}}\mathcal{P}_{N,0}$, $\mathcal{L}_N= \mathcal{R}_N\widetilde{\mathcal{L}}\mathcal{P}_{N,0}$, and $A_N(\lambda)(\beta;\gamma)=(\mathcal{B}x_{N,0}(\cdot;\lambda,\beta,\gamma); \ \widetilde{\mathcal{L}}x_{N,0}(\cdot;\lambda,\beta,\gamma))$.
As a consequence, the convergence analysis should also account for error
$$
\widetilde{Q}(B):= \sup \left\{ \frac{\|(\mathcal{L}-\widetilde{\mathcal{L}})x(\cdot;\lambda,\beta,\gamma) \|_X}{|(\beta;\gamma)|} \quad \colon \quad
\lambda \in \overline{B}, \ (\beta;\gamma) \in \mathbb{C}^{\bar{d}+d} \setminus\{0\} \right\}
% \widetilde{Q}(B):= \sup_{\lambda \in \bar{B}, (\beta;\gamma) \in \mathbb{C}^{\bar{d}+d} \setminus\{0\}} \frac{\|(\mathcal{L}-\widetilde{\mathcal{L}})x(\cdot;\lambda,\beta,\gamma) \|_X}{|(\beta;\gamma)|},
$$
and \eqref{lambda_error} should be replaced by
\begin{equation} \label{lambda_error_quad}
\max_{i=1,\dots,m}| \lambda^*-\lambda_{i,N}| =  O\left(D_N(B)^{1/m}\right)+ O\left(\widetilde{Q}(B)^{1/m}\right).
\end{equation}
In practice, it is often convenient to use quadrature formulas associated with $\Theta_N$ or $\Theta_{0,N}$, so that $\mathcal{L}$ is approximated by $\widetilde{\mathcal{L}}_N$ by Gauss--Laguerre (on the scaled Laguerre zeros) or Gauss--Radau--Laguerre (on the scaled Laguerre extrema plus zero) quadrature formulas respectively, with corresponding error $\widetilde{Q}_N(B)$. In this case, the convergence $\lambda_{i,N}\to\lambda^*$ follows from $\widetilde{Q}_N(B)\to 0$ as $N \to \infty$ and the convergence rate depends on the regularity of the integrands in suitable Sobolev spaces (see e.g. \cite[Section 5.1]{MastroianniBook} for Gaussian quadrature rules on $\mathbb{R}_+$).

%%%%%%%%%%%%%%%%%%%%%%%%
%%%%%%%%%%%%%%%%%%%%%%%%

\section{Numerical results}\label{S_Results}
We here present some numerical tests on linear and nonlinear iDEs to illustrate the error bounds obtained in \Cref{Th2}. For linear iDEs we compute the error on the characteristic roots and show convergence with exponential order for increasing values of $N$. For nonlinear iDEs, we perform equilibrium bifurcation analyses using the package MatCont (version 7p1) \cite{Matcont2008} for MATLAB to illustrate the flexibility of the approach. The numerical tests were performed using MATLAB (version 2023a).
The standard Laguerre collocation nodes and quadrature weights are computed using the publicly available MATLAB suites from \cite{G00} and \cite{Weideman2000}.
For ease of implementation, in the case of iDDEs we use one single entry to discretize the first element of a pair in the space $X$, hence $(\alpha,\psi) \in X$ is discretized by a vector of dimension $N+1$ (instead of dimension $\bar{d}N=2N$).
MATLAB codes to reproduce the tests are available at \url{https://github.com/scarabel/}.

\subsection{Linear iDEs}
Consider the linear iDDE
\begin{equation} \label{linear-DDE}
\frac{\diff y}{\diff t}(t) = a y(t) + \int_0^{+\infty} k(s) y(t-s) \diff s, \quad t \geq 0,
 \end{equation}
for $a \in \mathbb{R}$ and $k \in L^1(\mathbb{R}^+)$ with $ \int_0^\infty |k(s)| \mathrm{e}^{\rho^* s} \diff s < \infty$ for some $\rho^* > 0$.  
If $a+\int_0^{+\infty} k(s)  \diff s \neq 0,$ then $\bar{y}=0$ is the unique equilibrium of \eqref{linear-DDE}. The characteristic roots in $\mathbb{C}_\rho$ are the solutions of $\lambda = a + \widehat{k}(\lambda)$, where $\widehat{k}$ denotes the Laplace transform of $k$, and are eigenvalues of the linear operator $\mathcal{A}^{\text{DDE}} := \mathcal{A}_0^{\text{DDE}} + \mathcal{L}^{\text{DDE}}$, see \eqref{LADDE}. We study how they are approximated by the eigenvalues of the PSD $\mathcal{A}_N^{\text{DDE}}$.
The parameter values for the numerical tests are collected in \cref{t:linear-DE}.

\begin{table}[t]
\caption{Parameter values for the numerical tests on the linear iDDE \eqref{linear-DDE} and iRE \eqref{linear-RE}.}
\label{t:linear-DE}
\centering
\footnotesize{
\begin{tabular}{c lccccccl}
 \toprule
 & & kernel & $\mu$ & $a$ & $k_0$ & $\sigma$ & & $\lambda$ \\
 \midrule
(a1) & iDDE & exponential & 2 & 3 & $-6$ & - & & $\lambda=0$ \\
(a2) & iDDE & exponential & 2 & 3 & $-6$ & - & & $\lambda = a-\mu$ \\
(b) & iDDE & exponential & 2 & 2 & $-8$ & - & & $\lambda = \mu\, i$ \\
(c) & iDDE & exponential &  2 & 6 & $-16$ & - & & $\lambda = (a-\mu)/2$ (double) \\
(d) & iDDE & gamma & 4 & 0 & - & 2 & & $(\lambda-1)(\lambda+\mu)^2=\mu^2$ \\
(e) & iDDE & gamma & 4 & 0 & - & $\pi$ & & $\lambda = a + \widehat{k}(\lambda) $ \\
\midrule
%(f) & iRE & exponential & 2 & - & 2 & - & & $\lambda = 0 $ \\
%(g) & iRE & exponential & 2 & - & 5 & - & & $\lambda = k_0-\mu $ \\
(f) & iRE & non-monotonic & 1 & 1 & 1 & - & & \eqref{CE_sin} \\
(g) & iRE & non-monotonic & 1.5 & 1 & 3 & - & & \eqref{CE_sin} \\
\bottomrule
\end{tabular}
}
\end{table}
For an exponential kernel $k(s) = k_0 \mathrm{e}^{-\mu s}$, $\mu \in \mathbb{R}_+$, % and $\rho < \mu$, 
% the characteristic equation is 
% $\lambda- a - \frac{k_0}{\lambda + \mu}=0$,
it is easy to show that the characteristic roots exist in $\{\Re \lambda>-\mu\}$ and they can be written explicitly as 
$$\lambda = (a-\mu\pm \sqrt{(\mu+a)^2+4k_0})/2.$$
Note in particular that: $\lambda=0$ is a root if $k_0=-a\mu$; the roots are real if $k_0 \geq -(\mu+a)^2/4$, and have double multiplicity if $k_0 = -(\mu+a)^2/4$. 
In the numerical tests we take the two choices $\rho_1=\mu/2$ and $\rho_1=\mu/4$ to compare the impact of the nodes used for quadrature, and we fix $\rho=\rho_1$. To study the convergence of eigenvalues and eigenvectors, the eigenvector $\psi_{\lambda,N}$ of $\mathcal{A}_N^{\text{DDE}}$ is first normalized such that the first component equals one to obtain an approximation of $\psi_\lambda = \mathrm{e}^{\lambda \cdot}$.
% The latter is measured by normalising the eigenvector of $\mathcal{A}_N^{\text{DDE}}$ such that the first component equals one, and by computing the maximum weighted difference on the collocation nodes.
\cref{f:LinearDDE_conv} shows the log-log plot of the computed errors $|\lambda-\lambda_N|$ and $\|w(\psi_\lambda-\psi_{\lambda,N})\|_{\infty,N}$, where $\| \cdot \|_{\infty,N}$ indicates the discrete supremum vector norm (i.e., the maximum difference on the collocation nodes), as a function of $N$: for finite (polynomial) order of convergence we would expect to observe straight decreasing lines; an exponential (infinite) order of convergence translates into exponentially decreasing trends \cite[Section 2.4]{BoydBook}.
\cref{f:LinearDDE_conv} illustrates several features predicted by \cref{Th2}: the error decreases exponentially as $N$ increases; both larger modulus and larger real part imply slower convergence (a2, b); the maximal accuracy is halved for double roots (c); and the order of convergence depends on the quadrature error, so quadrature rules that are exact on the exponential kernel ($\rho_1=\mu/2$) guarantee faster convergence than non-exact rules ($\rho_1=\mu/4$). 
In the case of $\lambda=0$ (a1) the theoretical collocation error is zero if the quadrature is exact, and indeed the errors in both, eigenvalue and eigenvector, fall to machine precision already for $N=1$ when $\rho_1=\mu/2$, while the error coincides with the quadrature error, and the convergence is exponential, for $\rho_1=\mu/4$.
Moreover, for $\lambda=0$, Gauss-Laguerre quadrature rules and Gauss--Radau--Laguerre rules show a similar trend, while in the other tests the former seem to guarantee a faster exponential decay compared to the latter in general.

In tests (d) and (e) we considered a Gamma-distributed kernel $k(s)=f_{\mu}^{(\sigma)} (s)$, with rate parameter $\mu$ and shape parameter $\sigma$, i.e.,
\begin{equation}\label{Gamma-density}
f_{\mu}^{(\sigma)} (s) = \frac{\mu^\sigma s^{\sigma-1} \mathrm{e}^{-\mu s}}{\Gamma(\sigma)}, \quad s\geq 0.
\end{equation}
The reference eigenvalues in this case were obtained by solving algebraically the characteristic equation using the MATLAB solver \verb!fsolve!.
In this case, too, the choice of a suitable $\rho_1$ for quadrature rules can be linked to $\mu$, with $\rho_1=\mu/2$ giving faster convergence than $\rho_1=\mu/4$. 
The case of a non-integer shape parameter in \cref{f:LinearDDE_conv}(e) is particularly interesting, as it illustrates the role of the quadrature error in limiting the convergence order, as shown by \cref{lambda_error_quad}: the errors decay exponentially when the integrals are computed with the built-in MATLAB integral function (black lines, tolerance set at $10^{-12}$), while the convergence is of polynomial order (visible in the numerical errors decaying asymptotically as straight lines in the log-log diagram) when the integrals are approximated with the Gaussian quadrature rules associated to the collocation nodes (gray lines). 
The polynomial order of convergence in the latter case was estimated as $N^{\alpha}$ with $\alpha \approx -3.1479$.
\begin{figure}[t]
\centering
\includegraphics[width=\textwidth]{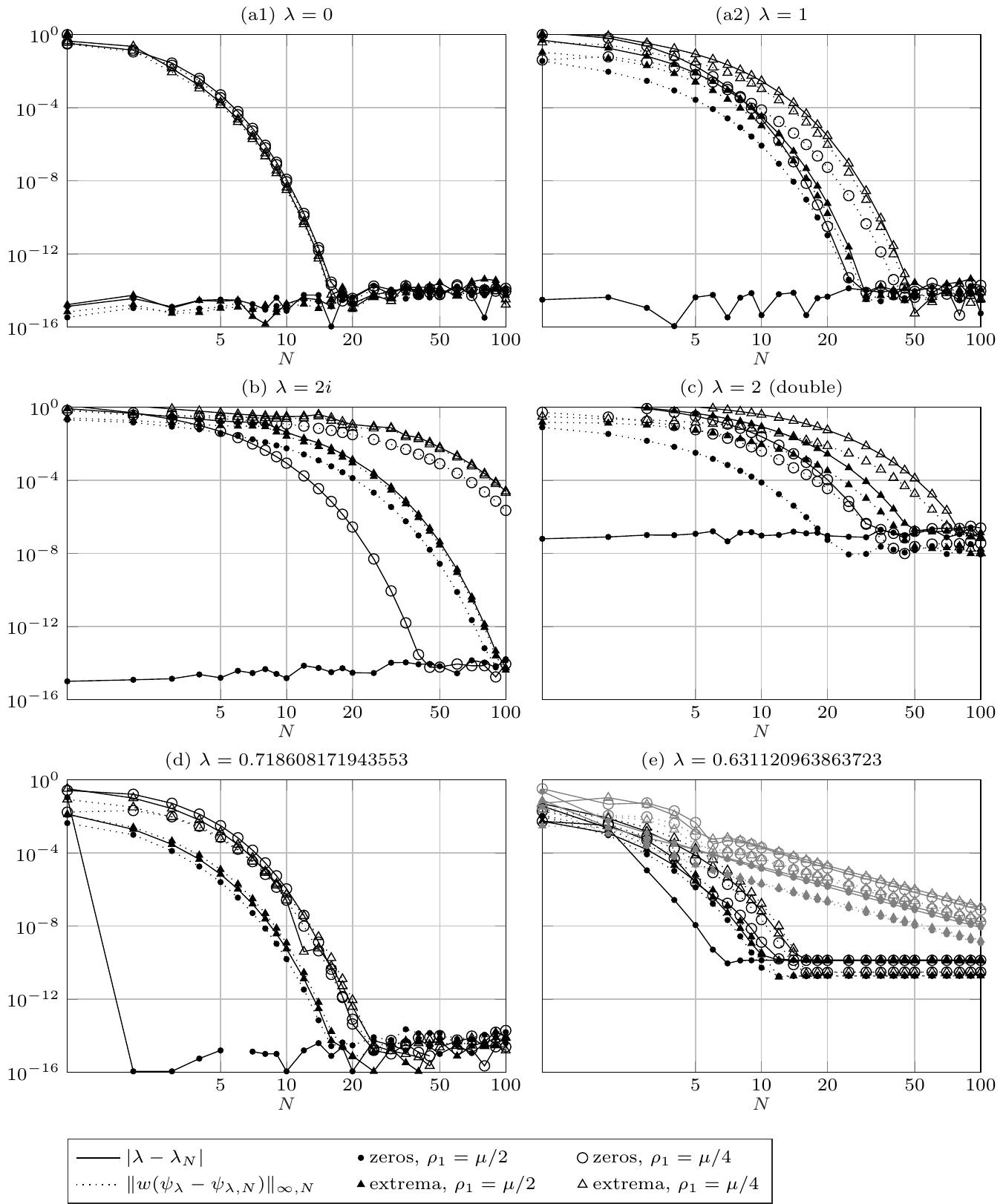}\\
\caption{
Linear DDE \eqref{linear-DDE}: 
log-log plot of $|\lambda-\lambda_N|$ (solid) and $\|w(\psi_\lambda-\psi_{\lambda,N})\|_{\infty,N}$ (dotted), for parameters specified in \cref{t:linear-DE}. See main text for further details.  
% Top: $\mu=2$, $a=1$, $k_0=-4.5$, hence $\lambda=-0.5\pm 1.5 i$; bottom: $\mu=2$, $a=1$, $k_0=-2.25$, hence $\lambda=-0.5$.
}
\label{f:LinearDDE_conv}
\end{figure}

Next, we consider the linear iRE 
\begin{equation}\label{linear-RE}
y(t) = \int_0^{+\infty} k(s) y(t-s) \diff s, \quad t > 0,
\end{equation}
where $k \in L^\infty(\mathbb{R}^+)$ with ${\rm ess} \sup_{s \in \mathbb{R}_+} |k(s)|\mathrm{e}^{\rho^* s} < \infty$ for some $\rho^* > 0$.  
If $\int_0^\infty k(s) \diff s \neq 0$, $\bar{y}=0$ is the unique equilibrium of \eqref{linear-RE}.
The characteristic roots in $\mathbb{C}_\rho$ are the solutions of $ 1 = \widehat{k}(\lambda)$, and are eigenvalues of the linear operator $\mathcal{A}^{\text{RE}} := \mathcal{A}_0^{\text{RE}} + \mathcal{L}^{\text{RE}}$, see \eqref{LARE}, with corresponding PSD $\mathcal{A}_N^{\text{RE}}$.

\cref{f:LinearRE_conv} shows the errors in the approximation of $\lambda$ and the corresponding eigenfunction $\psi_\lambda=j^{-1}(w\varphi_\lambda) = \mathrm{e}^{\lambda\cdot}$, in the discrete weighted $\infty$- and $1$-norm. 
To obtain a normalization comparable with the exact eigenfunction, we first applied the differentiation matrix to the discrete eigenfunction $\varphi_{\lambda,N}$ of $\mathcal{A}_N^{\text{RE}}$ and then normalized it so that the first component equals one, to obtain an approximation $\psi_{\lambda,N}$ of $\psi_\lambda$. 
For an exponential kernel, the convergence plots show very similar trends as for the iDDE, so we did not include them in the paper. Instead, we considered the non-monotone kernel $k(s) = k_0 \mathrm{e}^{-\mu s}(\sin(a s)+1)$, for $a,\mu, k_0 \in \mathbb{R}_+$, for which the characteristic roots (with $\Re\lambda>-\mu$) satisfy 
\begin{equation}\label{CE_sin}
 \lambda^3 + \lambda^2 ( 3 \mu -k_0) + \lambda ( 3 \mu^2  -2 k_0 \mu +a^2 - k_0 a ) + (\mu^2+a^2) (\mu- k_0) - k_0 a \mu = 0.
 \end{equation}
Parameter values are collected in \cref{t:linear-DE}.
In the numerical tests we constructed the collocation nodes and quadrature weights for two different choices, $\rho_1 = \mu/2$ and $\rho_1=\mu/3$, and we fixed $\rho = \mu>\rho_1$. The weighted 1-norm of the error is then computed with respect to $w(\theta) = \mathrm{e}^{\rho\theta}$. 
The convergence trends in \cref{f:LinearRE_conv} highlight similar behavior as for iDDEs, consistent with \cref{Th2}.

\begin{figure}[t]
\centering
\includegraphics[width=\textwidth]{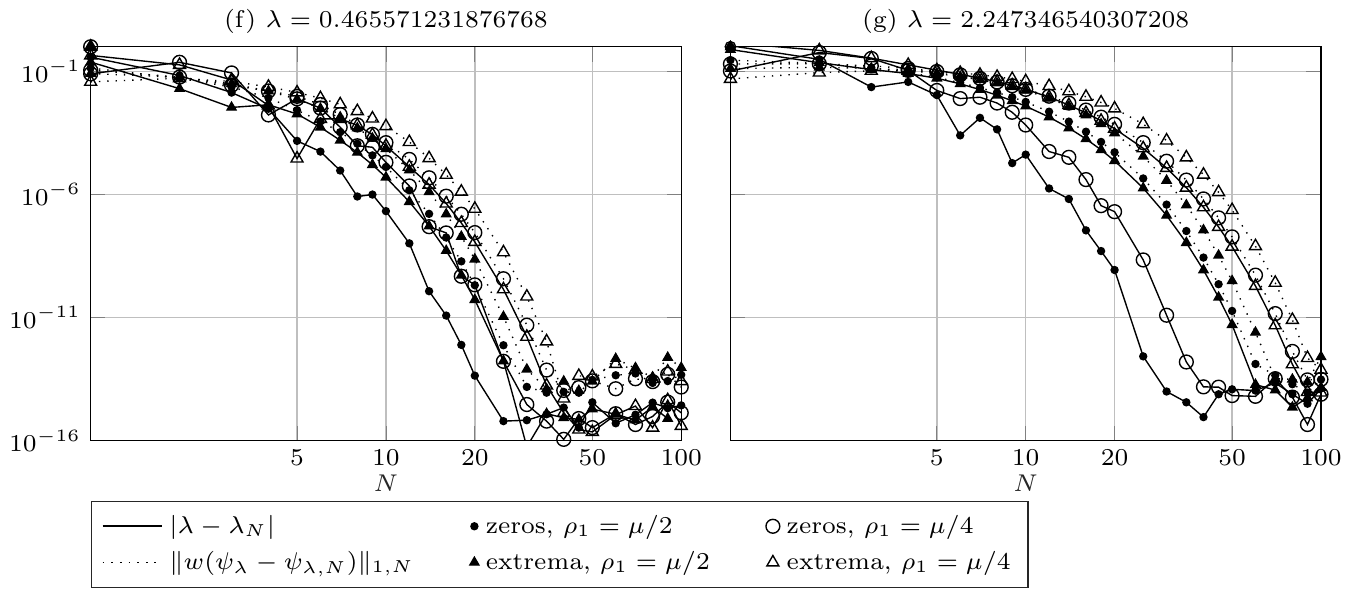}
\caption{
Linear RE \eqref{linear-RE}: 
log-log plot of $|\lambda-\lambda_N|$ (solid) and $\|\psi_\lambda-\psi_{\lambda,N}\|_{1,N}$ (dotted), for parameter values specified in \cref{t:linear-DE}.
}
\label{f:LinearRE_conv}
\end{figure}

\subsection{Numerical bifurcation analysis of nonlinear iDEs}
We now consider nonlinear iDEs and perform the numerical bifurcation analysis of equilibria using the numerical continuation package MatCont applied to the ODE system \eqref{ODE}. For each example below we specify the parameters used to construct the ODE system and perform the continuation. The bifurcation diagrams show the output of the numerical continuation.
Although Laguerre zeros appear to ensure better accuracy than Laguerre extrema in the linear examples (\cref{f:LinearDDE_conv} and \ref{f:LinearRE_conv}), numerical instabilities caused by the eigenvalues with large modulus 
%(see, e.g., \cite{FS94}) 
seem to substantially affect the computations in the nonlinear case. For this reason we use Laguerre extrema and the corresponding Gauss--Radau--Laguerre quadrature rules.

Consider the nonlinear iDDE taken from \cite{BB16} 
\begin{equation} \label{beretta-breda}
\frac{\diff y}{\diff t}(t) = -\delta_A y(t) + b \int_0^{+\infty} f_{m/\tau}^{(m)} (s) \mathrm{e}^{-\delta_J s - a y(t-s)} y(t-s) \diff s, \quad t\geq 0,
\end{equation}
where $\tau,m,a,b,\delta_A,\delta_J>0$ and $f_{\mu}^{(m)} (s)$ is the Gamma distribution \eqref{Gamma-density}.
When $m \in \mathbb{N}$, \eqref{beretta-breda} admits an equivalent ODE representation via the linear chain trick \cite{GSV18}, which we use as reference to compare the bifurcation diagrams computed via the PSD.
If $m$ is large enough, the positive equilibrium undergoes two Hopf bifurcations when varying $\tau$.
The Hopf bifurcation curve in $(\tau,m)$, which separates the parts of the plane where the positive equilibrium is either stable or unstable, is shown in \cref{f:beretta-breda} (top-left panel), showing also that the curves computed for $N=10$ and $N=20$ are indistinguishable in the figure.
Note that, since $y$ appears as an exponent in \eqref{beretta-breda}, to obtain reliable results we used the positive part of the state vector when computing the right-hand side of the PSD.
\cref{f:beretta-breda} shows the convergence of the observed error in the approximation of the two Hopf bifurcation points and the corresponding eigenvalues for $m=6.5,\ 7,\ 7.5$. 
Note that the discretization approach allows to consider distributions with non integer shape parameter $m$, offering a very general approach for iDDEs and iREs.

\begin{figure}[t]
\centering
\includegraphics[width=\textwidth]{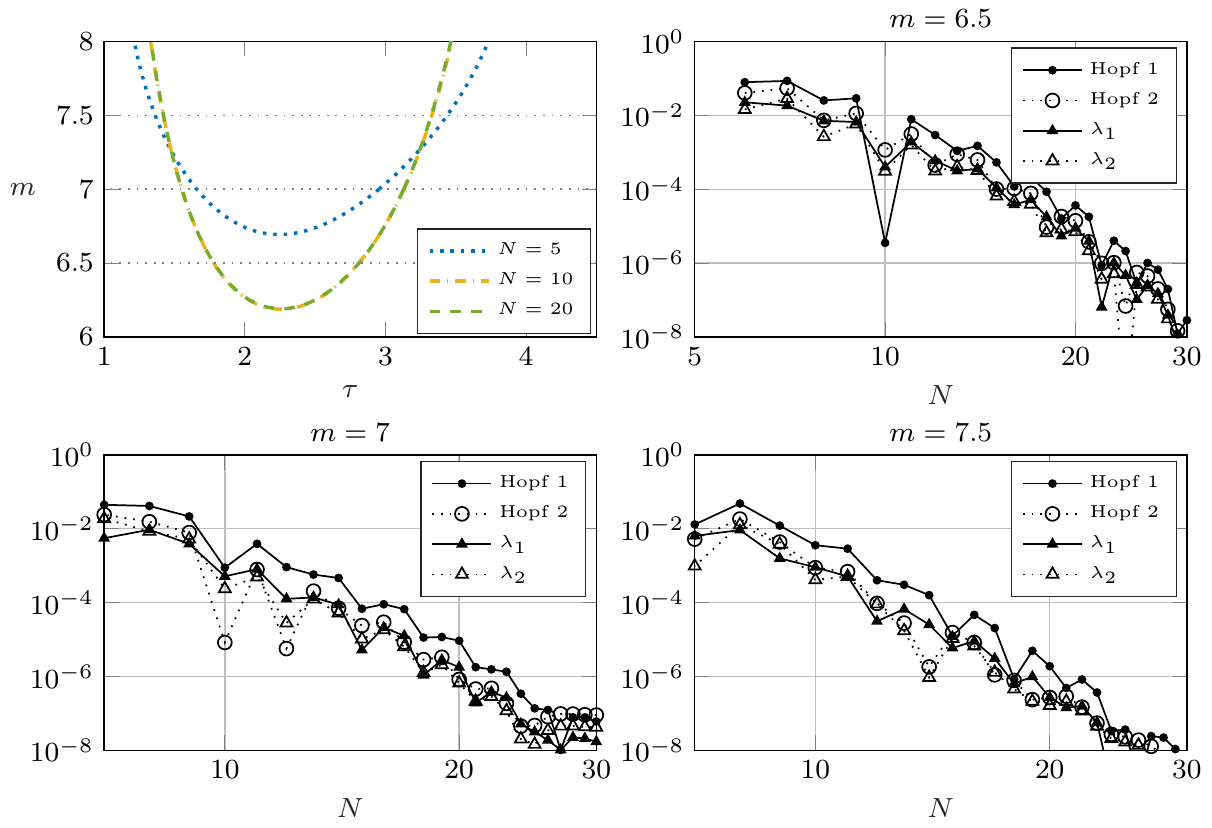}
\caption{
Top-left: Hopf bifurcation curve of \eqref{beretta-breda} in $(\tau,m)$ computed with MatCont, for $\delta_A=0.5,\,\delta_J=1,\,a=7,\,b=350$, and $\rho=(\delta_J+m/\tau)/4$, for $N=5,10,20$ (stable positive equilibrium below the curve). Note that the curves for $N=10,\,20$ are indistinguishable. 
The other panels show the error in the two Hopf bifurcation points and the critical characteristic root at Hopf, for $m=6.5,\,7,\,7.5$, with MatCont tolerance $10^{-10}$. Reference Hopf values are computed from the equivalent ODE formulation \cite{GSV18}.
}
\label{f:beretta-breda}
\end{figure}

% For m = 7.5
%Lambda1 = 0.000000000011888 - 0.920891288265355i
%Lambda2 = -0.000000000002669 - 1.600788394913181i
%H1 =   3.326878804666390
%H2 =   1.421993119543206

% For m = 6.5
%Lambda1 =  -0.000000000007983 - 1.058937041506656i
%Lambda2 = 0.000000000007068 - 1.409481346696375i
%H1 =   2.812428940518387
%H2 =   1.779141909897652

%%%%%%%%%%%%%%%%%%%%%%%%%%%%%%%%%%%%%%%%%%%%%%%%%%%%%%%%%

\begin{figure}[t]
\centering
\includegraphics[width=\textwidth]{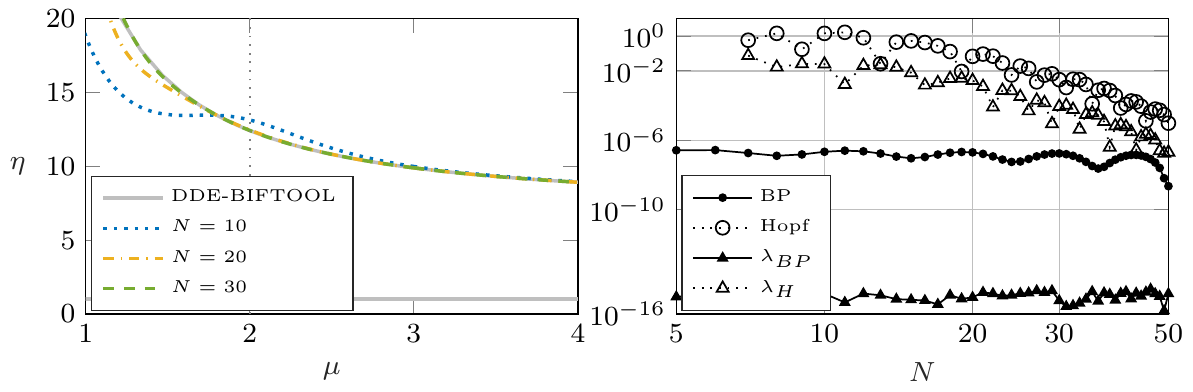}
\caption{
Left: Hopf bifurcation curve of \eqref{Nich-RE} in $(\mu,\gamma/\mu)$, for $\eta:=\beta_0 \mathrm{e}^{-\mu}/\mu$, computed using MatCont tolerance and $10^{-6}$, $\rho=\mu/2$. The positive equilibrium exists above the line $\eta=1$ (`branching point', BP). Gray lines are computed with DDE-BIFTOOL on the equivalent DDE formulation.
Right: error in BP and Hopf for $\mu=2$, with MatCont tolerance $10^{-10}$. Reference values are computed with $N=50$.
}
\label{f:blowflies}
\end{figure}

Consider now a nonlinear iRE equivalent to the Nicholson's blowflies DDE \cite{Gurney1980}
\begin{equation}\label{Nich-RE}
y(t) = \beta_0 \, \mathrm{e}^{-\int_1^{\infty} y(t-s) \mathrm{e}^{-\mu s} \diff s} \, \int_1^{\infty} y(t-s) \mathrm{e}^{-\mu s} \diff s, \quad t>0,
\end{equation}
for $\beta_0,\mu \geq 0$, where $y(t)$ represents the total birth rate of a population where individuals become adult after a maturation delay which is normalized to 1.
Equation \eqref{Nich-RE} admits a unique nontrivial equilibrium $\overline{y} = (\log \frac{\beta_0}{\mu} - \mu) \mu \mathrm{e}^{\mu}$ for $\beta_0 > \mu \mathrm{e}^{\mu}$, which undergoes a Hopf bifurcation as $\beta_0$ increases, and, if $\mu$ is large enough, the periodic solution undergoes a period doubling bifurcation \cite{DeWolff21}. 
The stability regions of the positive equilibrium in a two-parameter plane are shown in \cref{f:blowflies} (left), compared to those obtained by applying DDE-BIFTOOL to the equivalent DDE with finite delay (for the equivalence see \cite{DeWolff21}). The numerical error in the approximation of the transcritical and Hopf bifurcations with respect to $\beta_0$ are illustrated in \cref{f:blowflies} (right), for $\mu=4$. 
For implementation purposes, the right-hand side of \eqref{Nich-RE} was reformulated integrating by parts, and then computed with respect to the integrated state $Y(t) := \int_0^t y(s)\diff s$.
%as
%$$ \int_1^{\infty} y(t-s) \mathrm{e}^{-\mu s} \diff s 
%% = \left[ - Y(t-s) \mathrm{e}^{-\mu s} \right]_1^\infty -\mu \int_1^\infty Y(t-s) \mathrm{e}^{-\mu s} \diff s 
%= Y(t-1) \mathrm{e}^{-\mu} -\mu \int_1^\infty Y(t-s) \mathrm{e}^{-\mu s} \diff s,
%\quad Y(t) := \int_0^t y(s)\diff s.
%$$ 
The quadrature rules on the interval $[1,+\infty)$ were obtained by appropriately shifting the standard Gauss--Radau--Laguerre quadrature nodes. 
We note that the eigenvalue at the branching point ($\lambda=0$) is approximated to machine precision already with $N=1$, consistently with the analysis for linear equations. The approximation of the bifurcation point loses some accuracy probably due to the nonlinear solver used by MatCont for evaluating the test functions. The eigenvalues at the Hopf bifurcation point and the bifurcation point itself show a slower convergence, consistently with the fact that, at the Hopf bifurcation point, the eigenvalues are imaginary and therefore have strictly positive modulus.

\section{Conclusions}\label{S_Concluding remarks and outlook}
In this paper we have introduced a general abstract framework for the pseudospectral reduction of nonlinear iDEs to a finite-dimensional system of ODEs, which can be studied with well-established software for ODEs (specifically we used MatCont for MATLAB).
For iREs we have proposed a novel approach compared to \cite{GSV18}, exploiting the equivalence via integration proposed for finite delay in \cite{SDV21}. In this way, the PSD returns directly an ODE system without involving an implicit algebraic condition.

For linear(ized) iDDEs and iREs we have proved the convergence of the PSD approximation of the characteristic roots to the true ones, answering an important question raised when the method was first proposed in \cite{GSV18}. The convergence holds for a general weighted $L^p$-norm, therefore also in the Hilbert space $L^2_\rho$. The numerical tests illustrate the theoretical result of convergence of characteristic roots and eigenfunctions in norm $\|\cdot\|_{\infty,\rho}$ for iDDEs and $\|\cdot\|_{1,\rho}$ for iREs. 
The theoretical error bound depends on the modulus, the real part, and the multiplicity of the target eigenvalue, as well as the discretization index $N$, the choice of the family of collocation nodes (chosen as the zeros or extrema of the scaled Laguerre polynomials, orthogonal with respect to a weight $\mathrm{e}^{\rho_1 \cdot}$, $\rho_1<\rho$) and the quadrature formula used to approximate the integrals. 
In particular, approximating eigenvalues with larger modulus with a given accuracy requires larger discretization indices. 
Despite this, since the right half-plane $\mathbb{C}_\rho$ contains a finite number of eigenvalues, our theoretical result guarantees that all eigenvalues in $\mathbb{C}_\rho$ can be approximated with desired accuracy by choosing $N$ large enough.
%In particular, the error bound approaches one for eigenvalues with increasing modulus, requiring larger discretization indices $N$ to obtain reliable approximations. Importantly, the theoretical result suggests that larger (in modulus) eigenvalues could be missed by an approximation with fixed index $N$, but does not characterize the location of those eigenvalues. 
However, determining such value of $N$ is a difficult problem, unless one can estimate \emph{a priori} the location of the spectrum. %which our estimates cannot resolve. 
To investigate this problem experimentally, we performed extensive numerical tests, which confirm that the observed accuracy strongly depends on the specific eigenvalue of interest and the chosen quadrature formula. For optimal choices, discretization indices of the order of 20--30 usually guarantee machine precision for eigenvalues with multiplicity one.
%
%\textcolor{magenta}{We have finite eigenvalues quindi teoricamente li approssimiamo tutti per $N$ sufficientemente grande. Choice of $N$ is a difficult problem if we don't have an estimate of $\lambda$. La sperimentazione numerica serviva anche a questo. }

Numerical tests on nonlinear equations demonstrate the flexibility of the method to study the bifurcation of equilibria.
% While in general it is not possible to exclude the presence of large eigenvalues that are not captured by the approximation of index $N$, 
To verify the correctness of the computed bifurcation diagrams, in this paper we have considered equations in which the bifurcation diagram can be obtained via an alternative route, hence allowing to compare the PSD output with reference results: the nonlinear iDDE \eqref{beretta-breda} admits an equivalent ODE formulation, at least for integer shape parameters of the kernel, whose bifurcation properties can be studied independently with MatCont; and the iRE \eqref{Nich-RE} admits an equivalent formulation as a DDE with finite delay that can be studied with bifurcation software for DDEs (here we used DDE-BIFTOOL for MATLAB).

Future work includes the proof of convergence of the resolvent operators, that may be useful in turn to prove the convergence of the solution operators of the initial value problems and of the periodic solutions. 
From a computational point of view, we plan to investigate numerically the impact of different sets of collocation nodes and use truncated interpolation and quadrature schemes to improve computation times in the spirit of \cite{MM082}.
While the abstract formulation encompasses other types of models, the convergence proof in this paper relies on several assumptions specific to iDEs, including the form of the boundary condition defined by $\mathcal{B}$ and that the operator $\mathcal{F}$ maps on a $d$-dimensional space of functions $w\Xi_d$, on which weighted interpolation is exact. 
One could consider extending the techniques to more general boundary conditions in the form $\mathcal{B}x=0$. 
For more general basis functions of $\Xi_d$ or different (e.g., truncated) interpolation schemes, the analogous of \cref{Th00} would require the assumption that the interpolation scheme is convergent on the functions in $\Xi_d$. Similarly, the one-to-one correspondence of equilibria and linearization (\cref{th:one-to-one}) may hold only in an approximate sense, if interpolation is not exact on the equilibria.
Finally, the analysis conducted in the appendix also provides the exact location of the spectrum of the reduced differentiation matrix, i.e., the differentiation matrix with zero boundary condition, associated with the zeros and extrema of the standard Laguerre polynomials on the positive half-line. 
Understanding the behavior of the spectrum of differentiation matrices is a notoriously difficult problem, and these results add a first significative contribution to the study of such matrices, which are commonly used in the numerical solution of differential equations. We aim to perform further explorations in future work.

\appendix
\section{General results for a collocation problem on $\mathbb{R}_+$} \label{S_appendix}
In this section we assume $d=1$ for simplicity and, in order to work with the family of the standard Laguerre polynomials $\{L_N\}_{N \geq 0}$, orthogonal with respect to the weight $\mathrm{e}^{-t}$ on the positive half-line $\mathbb{R}_+:=[0,+\infty)$, we consider the equivalent reformulation of the collocation problem on $\mathbb{R}_+,$ and we study the existence and uniqueness and the convergence of the collocation polynomials for two relevant choices of the collocation points: the zeros and the extrema of $L_N$, for which we refer for instance to \cite[Chapter 2.3.5]{MastroianniBook}. The same approach can be easily extended to other collocation points.

Given $\beta, c \in \mathbb{C}$ and $\mu \in \mathbb{C}$ with  $\Re \mu < 1/2$, let $y(t) = \mathrm{e}^{\mu t} \beta+\int_0^t  \mathrm{e}^{\mu (t-s)} c \diff s$
%= \mathrm{e}^{\mu t} \beta+(\mathrm{e}^{\mu t}-1)\frac{c}{\mu}$, $t \in \mathbb{R}_+$
 be the solution of the initial value problem 
\begin{equation}\label{ODE1}
\begin{cases} 
\frac{\diff y}{\diff t} = \mu \, y + c, \quad t \in \mathbb{R}_+,\\[5pt]
y(0)=\beta.
\end{cases}
\end{equation}

Note that, since $\Re \mu < 1/2$, for all $\delta \geq 0$ and $w_{\delta}(t):=\mathrm{e}^{-(1/2+\delta)t},\  t \in \mathbb{R}_+,$ we have  $w_{\delta} y \in L^p(\mathbb{R}_+,\mathbb{C})$ for $1 \leq p \leq +\infty.$ 

Let $\Pi_{N}$ denote the set of algebraic polynomials of degree at most $N$ on $\mathbb{R}_+$. Given a set of distinct points $\{ t_{i,N} \}_{i=1}^N$ on $\mathbb{R}_+,$ the collocation polynomial $p_{N} \in \Pi_{N}$ for \eqref{ODE1} is defined by 
\begin{equation}\label{ODEp}
\begin{cases}
\frac{\diff p_{N} }{\diff t} = \mu \,p_{N} + c, \quad  t \in \{ t_{i,N} \}_{i=1}^N,\\[5pt]
p_{N} (0)=\beta.
\end{cases}
\end{equation}

Since $\frac{\diff p_{N}}{\diff t}- \mu \,p_{N} - c$ is a polynomial of degree $N$ which vanishes at the collocation points, the collocation equation \eqref{ODEp} admits the following equivalent representation
\begin{equation}\label{ODEd}
\left\{\begin{array}{l}
\frac{\diff p_{N} }{\diff t} - \mu \, p_{N} = c + k_N q_N, \quad t \in \mathbb{R}_+,\\[5pt]
p_{N}(0)=\beta,
\end{array}
\right.
\end{equation}
where
$q_N(t)=\prod_{i=1}^N(t-t_{i,N})$ and $k_N$ is a suitable constant. By expressing the polynomials in \eqref{ODEd} as follows
\begin{equation*}% \label{poly}
\frac{\diff p_{N} }{\diff t}(t)= \sum\limits_{j=0}^{N-1}a_j t^{j}, \quad  
p_{N}(t)=\beta+\sum\limits_{j=1}^{N}\frac{a_{j-1}}{j}\, t^{j}, \quad \text{and } \quad q_N(t)= \sum_{j=0}^N \frac{b_j}{j!} t^j,%q_N(t)= \sum_{j=0}^N b_j t^j,
\end{equation*}
from \eqref{ODEd} we derive the following recurrence relations
\begin{equation}\label{rec}
\left\{\begin{array}{l}
a_0=\mu \beta +c + k_N b_0\\
a_{j}-\frac{\mu}{j}a_{j-1}=k_N\frac{b_j}{j!} ,  \quad j=1,\ldots,N,\\
a_{N}=0.
\end{array}
\right.
\end{equation}
Now, by easy computations, we get from \eqref{rec} 
\begin{equation*} % \label{solution}
a_j= \frac{\mu^{j}}{j!}(\mu\beta +c) + k_Nd _{j}, \quad j=1,\ldots,N,
\end{equation*}
with $d_j$ given by 
\begin{equation}\label{dj}
%d_j = \frac{\mu^{j}}{j!}\sum\limits_{k=0}^{j} \mu^{-k} k! \, b_k, \qquad j=0,1,\dots, N.
d_j j!  = \sum\limits_{k=0}^{j} \mu^{j-k}  \, b_k, \qquad j=0,1,\dots, N.
\end{equation}
Finally, by requiring that $a_{N}=0,$ we get the following equation for $k_N$ 
\begin{equation}\label{kN}
\mu^{N}(\mu\beta +c) + k_Nd _{N} N!=0.
 \end{equation}
%\added{Note here that $k_N=0$ when $\mu=0$ for all $N\geq 1$, which is also evident from the fact that the solution $y$ of \eqref{ODE1} is a linear function and is interpolated exactly by the solution of \eqref{ODEp} for every $N\geq 1$. Moreover, each $d_j$}
Note that each $d_j$ in \eqref{dj} depends on $\mu$ and, through the polynomial $q_N,$ on the collocation points. In the following we focus on two relevant choices: the zeros and the extrema of the standard Laguerre polynomial $L_N$. For both, we study the existence and uniqueness, and the convergence w.r.t. the weight function $w_\delta$ in $L^p(\mathbb{R}_+,\mathbb{C})$, namely the behavior of $\| w_\delta(p_{N}-y)\|_p$ as $N \to \infty.$ 

Let $\hat e_N:= w_\delta(p_N-y).$ From \eqref{ODE1} and \eqref{ODEd} we easily get 
\begin{equation}\label{eN}
 \hat e_{N}(t)=  k_N  w_\delta(t)  \int_0^t \mathrm{e}^{\mu(t-s)}q_N(s) \mathrm{d}s , \quad t \in \mathbb{R}_+.
\end{equation}

%In what follows we first find the condition that guarantees the existence and uniqueness of the solution of \eqref{ODEp} for fixed $N\geq 1.$ Then, we study the convergence of $w_{\delta} p_N$ to $w_{\delta} y$ as $N \to \infty$ in $L^p$ for the collocation points given by the zeros or extrema of the classical Laguerre polynomials. 
%\begin{lemma}[Existence and uniqueness] 
%\label{lemma:existence-collocation}
%If \textcolor{red}{Non funziona perche' $b_j$ non defined}
%\begin{equation} \label{DN}
%d_{N}:=\frac{\mu^{N}}{N!}\sum\limits_{j=0}^{N} \mu^{-j} j! b_j
%\end{equation}
%%is non-zero, 
%Let 
%\begin{equation}\label{dj}
%d_j = \frac{\mu^{j}}{j!}\sum\limits_{k=0}^{j} \mu^{-k} k! \, b_k, \qquad j=0,1,\dots,N.
%\end{equation}
%If $d_N\neq 0$, 
%there exists a unique solution $p_N \in \Pi_N$ of \eqref{ODEd}, that can be written as
%\begin{equation*}% \label{pN}
%p_{N}(t)= \sum\limits_{j=0}^{N} \frac{\mu^{j-1}}{j!} (\mu\beta+c) \textcolor{red}{\, t^j} + k_N\sum\limits_{j=1}^{N}\frac{d_{j-1}}{j} \,t^{j}, \quad   t \in \mathbb{R}_+,
%\end{equation*}
%where 
%\begin{equation}\label{kN} 
%k_N:=\frac{- \mu^{N}}{d_N N!}(\mu \beta +c ).
%\end{equation}
%\end{lemma}
%
%\begin{proof}
% For both choices, we also study the convergence of $w_\delta p_{N} \to w_\delta y$ as $N \to \infty$ in $L^p$.
%Let $\hat e_N:= w_\delta(p_N-y).$ From \eqref{ODE1} and \eqref{ODEd} we easily get 
%\begin{equation}\label{eN}
% \hat e_{N}(t)= w_\delta(t) 
% \int\limits_0^t \mathrm{e}^{\mu(t-s)} k_N q_N(s) \mathrm{d}s , \quad t \in \mathbb{R}_+.
%\end{equation}

\begin{proposition}[Existence, uniqueness, convergence for Laguerre zeros]\label{pr:existence_zeros}
Let $\mu \in \mathbb{C}$, $\Re \mu < 1/2,$ and let $\{ t_{i,N} \}_{i=1}^N$ be the zeros of the $N$-degree standard Laguerre polynomial $L_N.$ Then there exists a unique polynomial solution $p_N$ of \eqref{ODEp} for all $N \geq 1.$ Moreover, defining 
\begin{equation}\label{Cl} 
C(\mu):=\frac{\mu}{\mu-1}, 
\end{equation}
we have 
\begin{equation}\label{kNqN-zeros}
k_N q_N(t)=-C(\mu)^N L_N(t) \, (\mu \beta+c), \quad t \in \mathbb{R}_+,
\end{equation}
and the following error bound 
 \begin{equation}\label{eNLX}
 \|\hat e_{N}\|_p \leq  |C(\mu)|^N  \frac{ (|\mu| |\beta| + |c|) K_{\delta}^p}{(1/2-\Re \mu)},
\end{equation}
where
\begin{equation}\label{Kpdelta}
K_{\delta}^p = 
\begin{cases}
	1, &p=\infty, \,   \delta \geq 0,\\%[2mm]
  	(p \delta)^{-1/p}, &1 \leq p <\infty, \, \delta > 0.%\\[2mm]
\end{cases}
\end{equation}
The convergence $\|\hat e_{N}\|_p \to 0$ as $N\to \infty$ follows for $p=\infty$ with $\delta \geq 0,$ and for $\ 1 \leq p <+\infty$  with $\delta > 0.$ 
\end{proposition}
\begin{proof}
For the standard Laguerre polynomial $L_N$ it holds \begin{equation*}
q_N(t)=(-1)^N N! \, L_N(t) = (-1)^N N! \, \sum\limits_{j=0}^N (-1)^j {N \choose j}  \,\frac{t^j}{j!},
\end{equation*}
see \cite[Chapter 2.3.5]{MastroianniBook}.
Hence, $b_j=(-1)^N N! (-1)^j {N\choose j}$, and, from \eqref{dj},
\begin{equation} \label{dN-zeros}
d_{N}
=(-1)^{N}\sum\limits_{j=0}^{N} \mu^{N-j}(-1)^{j} {N \choose j} % = (-1)^N(\lambda-1)^N 
= (1-\mu)^N \neq 0
\end{equation}
since $|1-\mu| \geq |1-\Re\mu| > \frac{1}{2}.$ Then the existence and uniqueness of the collocation solution follows for all $N \geq 1.$ From \eqref{kN} and \eqref{Cl} we obtain $k_N=-\frac{C(\mu)^N}{(-1)^N N!} (\mu \beta+c)$ and \eqref{kNqN-zeros}.

Let us study now the error $\hat{e}_N$. From \eqref{eN} and \eqref{kNqN-zeros} we get
\begin{equation*}% \label{eNL}
 \hat e_{N}(t)= C(\mu)^N   \, \mathrm{e}^{-\delta t}\int_0^t \mathrm{e}^{(\mu-1/2)(t-s)} \mathrm{e}^{-s/2} L_N(s) \,\mathrm{d}s \, (\mu \beta+c), \quad t \in \mathbb{R}_+.
\end{equation*}
Since $\mathrm{e}^{-s/2}| L_N(s)| \leq L_N(0)=1$ (see \cite[Section 2.3.5]{MastroianniBook}), we get   
\begin{align*}
 |\hat e_{N}(t)| 
	% &\leq \textcolor{red}{ (|\mu| |\beta| + |c|)} \, C(\mu)^N \mathrm{e}^{-\delta t} \int_0^t \mathrm{e}^{(\Re\mu-1/2)(t-s)} \mathrm{d}s \notag \\
 	&\leq   \, |C(\mu)|^N \mathrm{e}^{-\delta t} \ \frac{1-\mathrm{e}^{(\Re \mu-1/2) t}}{1/2 -\Re \mu} \,(|\mu| |\beta| + |c|) , \quad t \in \mathbb{R}_+.%\label{boundeNL}
 \end{align*}
The converge follows since $|C(\mu)|<1$ for $\Re \mu<1/2$ and from the bounds
\begin{equation}\label{pbound}
\begin{array}{ll}
\| \mathrm{e}^{-\delta t}(1-\mathrm{e}^{(\Re \mu-1/2)t})\|_\infty \leq 1, &\text { for all } \delta \geq 0,\\[5pt]
 \| \mathrm{e}^{-\delta t}(1-\mathrm{e}^{(\Re \mu-1/2)t})\|_p \leq (p\delta)^{-1/p},
 & \text { for all } \delta >0 \text{ and  } 1 \leq p <\infty.
 \end{array}
\end{equation}
\end{proof}

\begin{proposition}[Existence, uniqueness, convergence for Laguerre extrema] \label{pr:existence_extrema}
Let $\mu \in \mathbb{C}$, $\Re \mu < 1/2,$ and let $\{ t_{i,N} \}_{i=1}^N$ be the extrema of the $(N+1)$-degree standard Laguerre polynomial, i.e., the zeros of $\frac{\mathrm{d}}{\mathrm{d}t}L_{N+1}$. Then there exists a unique polynomial solution $p_N$ of \eqref{ODEp} for all $N \geq 1.$ Moreover, for $C(\mu)$ defined as in \eqref{Cl} and 
\begin{equation*}%\label{Dl}
D_N(\mu):= \frac{C(\mu)^{N}}{1-C(\mu)^{N+1} },
\end{equation*}
we have that 
\begin{equation}\label{kNqN-extrema}
k_N q_N(t)=   D_N(\mu) \, \frac{\diff}{\diff t} L_{N+1}(t) \left(\frac{\mu\beta+c}{\mu-1} \right), \ t \in \mathbb{R}_+,
% k_N q_N(t)= (-1)^{N+1} \frac{\mu^N}{(\mu-1)^{N+1}-\mu^{N+1}} (\mu\beta+c) \frac{\diff}{\diff t} L_{N+1}(t), \ t \in \mathbb{R}_+
% k_N q_N(t)=\frac{-\mu}{N} \left(\frac{\mu}{\mu-1} \right)^{N-1} t \,\frac{\mathrm{d}}{\mathrm{d}t}L_N(t) \,(\mu \beta+c), \ t \in \mathbb{R}_+,
\end{equation}
and the following bound 
\begin{equation}\label{eNRLX}
 \|\hat e_{N}\|_p \leq |D_N(\mu)| \frac{(|\mu| |\beta| + |c|) K_{\delta}^p }{|\mu-1|} \left(2+ \frac{|\mu|}{1/2-\Re\mu}\right)
 	\end{equation}
with $K_{\delta}^p$ defined in \eqref{Kpdelta}. 
The convergence $\|\hat e_{N}\|_p\to 0$ as $N\to \infty$ follows for $p=\infty$ with $\delta \geq 0,$ and for $\ 1 \leq p <+\infty$  with $\delta > 0.$ 
\end{proposition}

\begin{proof}
For the polynomial $\frac{\mathrm{d}}{\mathrm{d}t}L_{N+1}$ it holds
\begin{equation*}
q_N(t) =(-1)^{N} N! \, L_N^{(1)}(t)  
	= (-1)^{N} N! \sum\limits_{j=0}^{N} \frac{(-1)^j}{j!} 
{ N+1\choose N-j} \, t^{j}
\end{equation*} 
(see \cite[Chapter 2.3.5]{MastroianniBook}). Hence $b_j = (-1)^{N} N! (-1)^{j} { N+1\choose j+1}$, from \eqref{dj} we get
\begin{align} \label{dN-extrema}
d_N
% &= (-1)^N \mu^N \sum_{j=0}^{N} (-\mu)^{-j} { N+1 \choose j+1} \\
= (-1)^{N+1} \mu^{N+1} \sum_{j=0}^{N} (-\mu)^{-(j+1)} { N+1\choose j+1}
= (-1)^{N+1} \left( (\mu-1)^{N+1} - \mu^{N+1} \right) \neq 0
\end{align}
since $\Re \mu < 1/2.$ 
Then the existence and uniqueness of the collocation solution follows for all $N \geq 1.$ 
From \eqref{kN} and from
$$ \frac{\mu^N}{(\mu-1)^{N+1}-\mu^{N+1}} = \frac{C(\mu)^N}{(\mu-1) - \mu C(\mu)^{N} } = \frac{C(\mu)^N}{(\mu-1)(1 - C(\mu)^{N+1}) },$$
we obtain
$$k_N=\frac{\mu^N (\mu\beta+c)}{d_N N!} %= \frac{\mu^{N+1}}{N! d_N} \beta + \frac{\mu^{N}}{N! d_N} c  
= \frac{(-1)^{N+1}}{N!} \left[ \frac{C(\mu)^{N+1}}{1-C(\mu)^{N+1}} \beta + \frac{C(\mu)^{N}}{1-C(\mu)^{N+1}} \frac{c}{\mu-1} \right]$$
and \eqref{kNqN-extrema}. 
Let us now consider the error $\hat{e}_N$. From \eqref{eN} and \eqref{kNqN-extrema} we get
\begin{align*}% \label{eNRL}
 \hat e_{N}(t)&=D_N(\mu)\, w_\delta(t) \int_0^t \mathrm{e}^{\mu(t-s)} \frac{\diff}{\diff s} L_{N+1}(s) \diff s \, \frac{\mu\beta+c}{\mu-1} \\
  &= D_N(\mu) \, w_\delta(t) \Big(L_{N+1}(t)-L_{N+1}(0) \mathrm{e}^{\mu t}+ \mu \int_0^t \mathrm{e}^{\mu(t-s)} L_{N+1}(s) \diff s \Big) \, \frac{\mu\beta+c}{\mu-1}
  %&= D_N(\mu) \, w_\delta(t)\left((L_{N+1}(t)-L_{N+1}(0)e^{\mu t})(\beta+\frac{c}{\mu})+ \int\limits_0^t \mathrm{e}^{\mu(t-s)} L_{N+1}(s) \diff s \,(\beta \mu+c)\right)\\
\end{align*}
Since $\mathrm{e}^{-s/2} | L_{N+1}(s) |\leq L_{N+1}(0)=1$ (see \cite[Chapter 2.3.5]{MastroianniBook}), we get   
\begin{equation*}%\label{boundeNRL}
% |\hat e_{N}(t)| &\leq  |D_N(\mu)| \mathrm{e}^{-\delta t} \left( \frac{2}{|\mu|} +  \int\limits_0^t \mathrm{e}^{(\Re\mu-1/2)(t-s)} \diff s \right)(|\beta||\mu|+|c|)
|\hat e_{N}(t)| \leq  |D_N(\mu)| \mathrm{e}^{-\delta t} \left( 2 + |\mu| \frac{1-\mathrm{e}^{(\Re \mu-1/2) t}}{1/2 -\Re \mu} \right) \left(\frac{|\beta||\mu|+|c|}{|\mu-1|}\right)
\end{equation*}
and \eqref{eNRLX} from \eqref{pbound}. 
Since $|D_N(\mu)| \to 0$ as $N \to \infty$, the convergence follows.
\end{proof}
As a final important remark, note that from \eqref{kN} with $\beta=0$ and $c=0$ we get the equation $d_N=0$, whose roots $\mu \in \mathbb{C}$ are the eigenvalues of the reduced differentiation matrix (with zero boundary condition). In particular, for Laguerre zeros, the differentiation matrix has only the eigenvalue $\mu=1$ with multiplicity $N$ (see \eqref{dN-zeros}), while for Laguerre extrema, the $N$ eigenvalues are aligned along $\Re \mu = \frac{1}{2}$ (see \eqref{dN-extrema}). This result adds a new insight into the spectral properties of such matrices, which are very useful in the numerical solution of differential equations.

\section*{Acknowledgments}
The authors are members of CDLab (Computational Dynamics Laboratory, Department of Mathematics and Computer Science, University of Udine), of INdAM Research group GNCS, and of UMI Research group ``Modellistica socio-epidemiologica''. 
F.S.~acknowledges the QJMAM Fund for Applied Mathematics for travel funding that facilitated the completion of this paper. 
% Commented as already included in "Funding" in the first page
%The work of RS is supported by the Italian Ministry of University and Research (MUR) through the PRIN 2020 project (No. 2020JLWP23) ``Integrated Mathematical Approaches to Socio–Epidemiological Dynamics'' (CUP: E15F21005420006).
%The research of FS was partially supported by the UKRI through the JUNIPER modelling consortium (grant number MR/V038613/1).

%\bibliographystyle{siamplain}
%\bibliography{_biblio_all}

\end{document}